\documentclass[leqno,11pt,oneside]{amsart}
\usepackage{amsmath,amssymb,amsthm,mathrsfs}
\usepackage{epsfig,color}
\usepackage[latin1]{inputenc}
\usepackage[english]{babel}
\usepackage{mathtools}
\usepackage{braket}
\usepackage[toc]{appendix}
\usepackage{enumitem}
\usepackage{ulem}
\usepackage{graphicx}
\usepackage{xfrac, nicefrac}
\usepackage{csquotes}
\allowdisplaybreaks

\numberwithin{equation}{section}

\usepackage{esint}

\usepackage{hyperref}

\newcommand{\rn}{\mathbb{R}^n}

\def\A{\mathcal A}

\def\H{\mathcal H}

\def\R{\mathbb R}
\def\N{\mathbb N}

\newcommand{\dist}{\mathop{\mathrm{dist}}}
\def\e{\varepsilon}
\def\s{\sigma}

\def\vphi{\varphi}

\def\l{\lambda}

\def\dist{{\rm dist}}

\renewcommand{\a}{\alpha}
\renewcommand{\b}{\beta}
\renewcommand{\d}{\mathrm{d}}

\renewcommand{\l}{\lambda}

\newcommand{\ov}{\overline}

\newcommand{\diver}{{\rm div }}

\newcommand{\rme}{\vrho_{m,\e}}
\newcommand{\fme}{f_{m,\e}}
\newcommand{\ume}{u_{m,\e}}
\newcommand{\cme}{c_{m,\e}}
\newcommand{\umek}{u_{m,\e_k}}
\newcommand{\rmek}{\vrho_{m,\e_k}}
\newcommand{\fmek}{f_{m,\e_k}}

\def\B{\mathcal{B}}

\def\A{\mathcal{A}}
\def\d{\delta}
\def\vrho{\varrho}

\newtheorem*{theorem*}{Theorem}
\newtheorem{theorem}{Theorem}[section]
\newtheorem{lemma}[theorem]{Lemma}
\newtheorem{proposition}[theorem]{Proposition}
\newtheorem*{proposition*}{Proposition}

\theoremstyle{remark}
\newtheorem{example}[theorem]{Example}
\newtheorem{remark}[theorem]{Remark}
\newtheorem*{remark*}{Remark}

\title[Second order regularity for degenerate elliptic equations]{Second order regularity for degenerate $p$-Laplace type equations with log-concave weights}

\begin{document}

\begin{abstract}
We consider weighted $p$-Laplace type equations with homogeneous Neumann boundary conditions in convex domains, where the weight is a log-concave function which may degenerate at the boundary. In the case of bounded domains, we provide sharp global second-order estimates. For unbounded domains, we prove local estimates at the boundary. The results are new even for the case $p=2$.
\end{abstract}

\author{Carlo Alberto Antonini \textsuperscript{1}}
\address{\textsuperscript{1}Dipartimento di Matematica e Informatica ``Ulisse Dini'',
Universit\`a di Firenze,
Viale Morgagni 67/A, 50134
Firenze,
Italy}

\email{carloalberto.antonini@unifi.it, carlo.antonini@unimi.it}
\urladdr{}

\author{Giulio Ciraolo \textsuperscript{2}}
\address{\textsuperscript{2}Dipartimento di Matematica  ``Federigo Enriques'', Universit\`a di Milano, Via Cesare Saldini 50, 20133 Milano, Italy}
\email{giulio.ciraolo@unimi.it}
\urladdr{ }

\author{Francesco Pagliarin \textsuperscript{3}}
\address{\textsuperscript{3} Institut Camille Jordan, Universit\'e Claude Bernard Lyon 1, 43 boulevard du 11 Novembre 1918
69622 Villeurbanne cedex, France}
\email{pagliarin@math.univ-lyon1.fr}
\urladdr{}

\date{\today}

\subjclass[2020]{35B65, 35D30, 35J25, 35J62, 35J70, 35J92}
\keywords{Degenerate elliptic equations, Neumann problems, boundary regularity, second-order derivatives, convex domains, log-concave weights}

\maketitle

\section{Introduction}\label{intro1}

In this paper we are interested in second order regularity for solutions to weighted elliptic equations of the form
\begin{equation} \label{eq_main1}
-\mathrm{div}\big(\vrho(x) |\nabla u|^{p-2}\nabla u\big) = \vrho(x)\, f \quad\text{in $\Omega$\,,}
\end{equation}
under homogeneous Neumann boundary condition.
Here $p>1$, $\Omega$ is a convex domain of the Euclidean space $\R^n$, $n\geq 2$, and $\vrho:\Omega\to \R$ is a positive log-concave function which may vanish at the boundary. For bounded domains $\Omega$ we provide global estimates, while local estimates at the boundary will be given for unbounded domains.

When $\Omega$ is bounded, $p=2$ and $\vrho\equiv 1$, equation \eqref{eq_main1} reproduces the classical Poisson equation, for which second-order regularity is well-known, provided the boundary of $\Omega$ is smooth enough \cite[Chapters 8.3-8.4]{gt}, \cite[Chapter 3]{gris}. Precisely, there holds the equivalence
\begin{equation}\label{equiv:1}
    u\in W^{2,2}(\Omega)\iff f\in L^2(\Omega)\,.
\end{equation}

The generalization of such result for $p$-Laplace type equations, with $p\neq 2$ and $\vrho$ sufficiently regular and bounded away from zero, has been recently established in \cite{accfm,cia,miao, guarnotta}, and states that
\begin{equation}\label{equiv:2}
    |\nabla u|^{p-2}\nabla u\in W^{1,2}(\Omega)\iff  f\in L^2(\Omega)\,,
\end{equation}
together with quantitative estimates. We notice that, for convex domains $\Omega$, no additional assumption must be assumed on $\partial \Omega$. We also mention \cite{acf} and \cite{bv} for local results in anisotropic and RCD spaces, respectively, and \cite{akm, bdm, bsv, bcds, bcdks,cms, cia19, dks, DS, km, mms} for further lines of research concerning the regularity of solutions to $p$-Laplace type equations. 

It is clear that if $p=2$ then equivalence \eqref{equiv:2} reduces to \eqref{equiv:1}.
Thus, denoting by
\begin{equation*}
    \A(\nabla u)=|\nabla u|^{p-2}\nabla u
\end{equation*}
the so-called \textit{stress field}, \eqref{equiv:2} implies that $\A(\nabla u)$ encodes second order regularity of solutions to nonlinear equations of $p$-Laplacian type. We stress out the fact that the nondegeneracy of $\vrho$ (i.e., $\vrho>0$ in $\overline \Omega$) plays a fundamental role and for a general weight $\vrho \geq 0$ one cannot expect second order regularity of solutions even in the linear case.

In this paper, we deal with log-concave weights $\vrho$, i.e., a nonnegative function whose logarithm is concave. Therefore, we may write 
\begin{equation}\label{log:concave}
    \vrho(x)=e^{-h(x)}\,,\quad x\in \Omega\,,
\end{equation}
for some convex function $h: \Omega \to \R$. Moreover, if one takes a convex extension $h:\overline{\Omega}\to \R\cup\{+\infty\}$ and the corresponding weight $\vrho: \overline{\Omega}\to [0,+\infty)$, then then  the zero level set $\{\vrho=0\}$ coincides with $\{h=+\infty\}$.


Hence, given a log-concave weight $\vrho$ as above, we aim to show that
\begin{equation}\label{equiv:3}
   \A(\nabla u)=|\nabla u|^{p-2}\nabla u\in W^{1,2}(\Omega;\vrho)\iff f\in L^2(\Omega;\vrho)
\end{equation}
for solutions to \eqref{eq_main1} under homogeneous Neumann boundary conditions on $\partial \Omega$. Our goal is also to provide (possibly) sharp quantitative estimates.
 
Before describing our main results, it will be convenient to introduce some notation. For $q\geq 1$, we define the weighted Lebesgue space relative to $\vrho$ as
\begin{equation}
    L^q(\Omega;\vrho)=\Big\{v:\Omega\to\R\,\text{ Borel measurable}: \int_\Omega v^q\,\vrho\,dx<\infty \Big\}\,,
\end{equation}
endowed with the norm
\begin{equation}
    \|v\|_{L^q(\Omega;\vrho)}=\Big(\int_\Omega |f|^q\,\vrho\,dx\Big)^{1/q}\,.
\end{equation}
Accordingly, the weighted Sobolev space is defined by
\begin{equation}
    W^{1,q}(\Omega;\vrho)=\Big\{v\in L^q(\Omega;\vrho):\,\nabla v\in L^q(\Omega;\vrho)\Big\}\,,
\end{equation}
where $\nabla v$ is the distributional gradient of $v$. Standard arguments show that $W^{1,q}(\Omega;\vrho)$ is a Banach space when endowed with the norm
\begin{equation}\label{norm:sobrho}
    \|v\|_{W^{1,q}(\Omega;\vrho)}=\|v\|_{L^q(\Omega;\vrho)}+\|\nabla v\|_{L^q(\Omega;\vrho)}\,,
\end{equation}
which makes it uniformly convex, and in particular reflexive, when $q>1$. As usual, we denote by $q'=q/(q-1)$ the conjugate exponent of $q>1$.

Let us briefly review some results on weighted Sobolev spaces and quasilinear equations, whose literature is far too rich to be fully described. Classical references are the monographs \cite{kuf,kuf1,tur}, mostly dealing with weights satisfying the so-called $A_p$-condition, see also \cite{bdgp} and references therein for related and more recent results. For very general weights fulfilling certain capacitary conditions, and for which weighted Sobolev inequalities are valid, we refer to the monograph \cite{stred}.
Boundedness and H\"older continuity for solutions to weighted quasilinear equations are studied in \cite{mont,mont1}. Higher order H\"older regularity has been recently established in \cite{terracini,terracini1,terracini2} (when $p=2$) for weights which behave like $\dist(x,\Gamma)^a$ for $a>-1$, where $\Gamma$ is a regular hypersurface.

\medskip

In this paper, we study Neumann boundary value problems of the form
\begin{equation}\label{eq:neu}
    \begin{cases}
        -\mathrm{div}\big(\vrho\,|\nabla u|^{p-2}\nabla u \big)=\vrho\,f\quad & \text{in $\Omega$}
        \\
        \vrho\,\partial_\nu u=0\quad & \text{on $\partial \Omega$,}
    \end{cases}
\end{equation}
with compatibility condition
\begin{equation}\label{f:compat}
    \int_\Omega f\,\vrho\,dx=0\,.
\end{equation}
Here $\nu$ is the outward normal to $\partial \Omega$, and $\partial_\nu u$ is the normal derivative of $u$.

Equation \eqref{eq:neu} has to be interpreted in a weak sense. More precisely, we say that $u\in W^{1,p}(\Omega;\vrho)$
is a weak solution to \eqref{eq:neu} if
\begin{equation}\label{def:neu}
    \int_\Omega |\nabla u|^{p-2}\nabla u\cdot \nabla \vphi\,\vrho\,dx=\int_\Omega f\,\vphi\,\vrho\,dx\,,
\end{equation}
for all test functions $\vphi \in W^{1,p}(\Omega;\vrho)$. Moreover, equation \eqref{eq:neu} 
arises as the Euler-Lagrange equations for the energy functionals
\begin{equation} \label{functional}
\mathcal F_\vrho(v) = \frac{1}{p}\int_{\Omega} \vrho(x) |\nabla v|^p\,dx + \int_{\Omega} \vrho(x) f\,v\, dx \,,\quad v\in W^{1,p}(\Omega;\vrho)\,.
\end{equation}

Our first main result concerning Neumann problems is the following global regularity result.
\begin{theorem}\label{main:thmneu}
  Let $p>1$, $\Omega\subset \R^n$ be a bounded convex domain, $\vrho: \Omega\to (0,\infty)$ be a log-concave function on $\Omega$, and suppose $f\in L^2(\Omega;\vrho)\cap L^{p'}(\Omega;\vrho)$ satifies \eqref{f:compat}. Assume that $u\in W^{1,p}(\Omega;\vrho)$ is a weak solution to the Neumann problem \eqref{eq:neu}. Then
  \begin{equation}\label{thesis}
      \A(\nabla u)=|\nabla u|^{p-2}\nabla u\in W^{1,2}(\Omega;\vrho)\,,
  \end{equation}
 and there exist explicit constants $C_0=C_0(n,p)$ and
$C_1=C_1(n,p,d_\Omega,\vrho)$ such that 
  \begin{align}\label{stima}
     &\int_\Omega \big|\nabla\A(\nabla u)\big|^2\,\vrho\,dx  \leq C_0\,\int_\Omega f^2\,\vrho\,dx
     \\
    \label{stima2} & \int_\Omega \big|\A(\nabla u)\big|^2\,\vrho\,dx \leq C_1\,\int_\Omega f^2\,\vrho\,dx+C_1\,\Big(\int_\Omega |f|^{p'}\,\vrho\,dx \Big)^{2/p'}\,,
  \end{align}
  where $d_\Omega$ denotes the diameter of $\Omega$.
\end{theorem}

As we already mentioned, the novelty of Theorem \ref{main:thmneu} lies in the type of weight that we consider. As far as we know, the presence of a log-concave weight $\varrho$ which may vanish at the boundary (or on a portion of it) is not covered by the existing literature even in the case $p=2$.

 The estimate in \eqref{stima} shows the equivalence between the weighted $L^2$-norm of $\nabla \A(\nabla u)$ and that of the source term $f$, with quantitative constant $C_0$ depending only on $n$ and $p$. This estimate is sharp. Indeed in the case $p=2$ and $\vrho=1$ we obtain that $C_0=1$ which is the optimal constant for \eqref{stima}. 
  
 For what concerns \eqref{stima2}, we notice that in the case $p<2$ we need to assume $f\in L^{p'}(\Omega;\vrho)$ in order to obtain a uniform energy estimate for $\nabla u$, see Lemma \ref{lemma:en1} below. This additional integrability assumption compensates for the lack of energy estimates, and in particular the lack of a theory of generalized solution as in \cite[Theorem 3.8]{cia17}, for the weighted equation \eqref{eq:neu}.
 
We also mention that the estimates in Theorem \ref{main:thmneu} do not depend on the particular chosen extension of $\vrho$ to $\partial \Omega$. This shows a relevant robustness of the approach we adopted for this type of boundary value problems.

\medskip

Our next main results deal with Neumann problems \eqref{eq:neu} in unbounded convex domains $\Omega$. In this case, our results are local in nature, and we need to distinguish the singular case $1<p\leq 2$ and the degenerate one $p>2$.
We first provide some definitions: for $\Omega\subset \R^n$ convex possibly unbounded domain, we write
\begin{equation}
    L^q_{loc}(\overline{\Omega};\vrho)=\Big\{v:\Omega\to \R: v\in L^q(\Omega\cap K;\vrho),\,\text{for every compact set $K\subset \R^n$}\Big\}\,,
\end{equation}
and also
\begin{equation}
    W^{1,p}_{loc}(\overline{\Omega};\vrho)=\Big\{u:\Omega\to \R:\, u\in W^{1,p}(\Omega\cap K;\vrho), \,\text{for every compact set $K\subset \R^n$}\Big\}\,.
\end{equation}
We say that $u\in W^{1,p}_{loc}(\overline{\Omega};\vrho)$ is a local solution to 
\begin{equation*}
     \begin{cases}
        -\mathrm{div}\big(\vrho\,|\nabla u|^{p-2}\nabla u \big)=\vrho\,f\quad & \text{in $\Omega$}
        \\
        \vrho\,\partial_\nu u=0\quad & \text{on $\partial \Omega$,}
    \end{cases}
\end{equation*}
if it satisfies
\begin{equation}\label{eq:neuloc}
    \int_\Omega |\nabla u|^{p-2}\nabla u\cdot \nabla \vphi\,\vrho\,dx=\int_\Omega f\,\vphi\,\vrho\,dx\,,\quad\text{for all $\vphi\in C^\infty_c(\R^n)$},
\end{equation}
where $C^{\infty}_c(\R^n)$ is the set of smooth, compactly supported function in $\R^n$.
We remark that such test functions may not vanish on $\partial \Omega$.
\\

 We consider local estimates around points $x_0$ on the boundary of $\Omega$. Indeed, since $\vrho$ is strictly positive in $\Omega$, interior estimates are already well established \cite{cia,miao}. Differently from Theorem \ref{main:thmneu}, local estimates are obtained under the further assumption that $\vrho$ admits a continuous extension up to $\partial \Omega$. This is needed since in the approximation argument we exploit an extension theorem for the weighted Sobolev space which requires the continuity of the weight, see Lemma \ref{lem:tech}.

 For $1<p\leq 2$ we obtain the following theorem.

\begin{theorem}\label{thm:loc1}
    Let $\Omega\subset \R^n$ be a convex domain and let $\varrho: \Omega \to (0,+\infty)$ be a log-concave function which can be continuously extended to $\overline{\Omega}$. Let $1<p\leq 2$ and let $u\in W^{1,p}_{loc}(\overline{\Omega};\vrho)$ be a weak local solution to \eqref{eq:neuloc}, with $f\in L^{p'}_{loc}(\overline{\Omega};\vrho)$. Then there exists $R>0$  such that, for every $x_0\in \partial \Omega$ and $R_0\leq  R$, we have
\begin{equation}
    \A(\nabla u)\in W^{1,2}(\Omega\cap B_{R_0}(x_0);\vrho)\,,
\end{equation}
and there exist constants 
$C_0=C_0(n,p)>0$ such that
 \begin{equation}\label{stimaloc}
     \int_{\Omega\cap B_{R_0}(x_0)} \big|\nabla\A(\nabla u)\big|^2\,\vrho\,dx  \leq C_0\,\int_{\Omega\cap B_{2R_0}(x_0)} f^{2}\,\vrho\,dx+\frac{C_0}{R_0^2}\, \int_{\Omega\cap B_{2R_0}(x_0)}|\A(\nabla u)|^2\,\vrho\,dx 
     \end{equation}
     with
     \begin{equation}
\int_{\Omega\cap B_{2R_0}(x_0)} \big|\A(\nabla u)\big|^2\,\vrho\,dx \leq \left( \int_{\Omega\cap B_{2R_0}(x_0)}\vrho\,dx\right)^{(2-p)/p}\,\left(\int_{\Omega\cap B_{2R_0}(x_0)} |\nabla u|^{p}\,\vrho\,dx \right)^{2/p'}\,.
  \end{equation}

\end{theorem}

When $p>2$, we need to make the additional assumption
\begin{equation}\label{a:concc}
    \vrho(x)=\big[g(x)\big]^a\,,\quad\text{for a nonnegative concave function $g$ and a number $a\geq 0$.}
\end{equation}

Indeed, unlike the case $p\leq 2$, we cannot directly estimate the $L^2$-norm of $\A(\nabla u)$ via the $L^p$-norm of $\nabla u$. Instead, we rely on a suitable Poincar\'e inequality on annuli, see Lemma \ref{annuli:lemma} below, which is valid only for weights of the form \eqref{a:concc}, see Remark \ref{remark:annuli}.


\begin{theorem}\label{thm:loc2}
    Let $\Omega\subset \R^n$ be a convex domain, $\vrho$ be a nonnegative function satisfying \eqref{a:concc} which can be continuously extended to $\overline{\Omega}$. Let $p>2$ and $f\in L^2_{loc}(\overline{\Omega};\vrho)$, and suppose that $u\in W^{1,p}_{loc}(\overline{\Omega};\vrho)$ is a weak local solution to \eqref{eq:neuloc}. Then there exists $R>0$ such that, for every $x_0\in \partial \Omega$ and $R_0\leq R$, we have
\begin{equation}
    \A(\nabla u)\in W^{1,2}(\Omega\cap B_{R_0}(x_0);\vrho)\,,
\end{equation}
and there exist positive constants 
$C_0(n,p,\Omega,\vrho)$ and $\widehat{C}_0=\widehat{C}_0(n,L_\Omega)$, with $L_\Omega$ denoting the Lipschitz constant of $\Omega$, such that
 \begin{equation}\label{stimaloc1}
     \int_{\Omega\cap B_{R_0}(x_0)} \big|\nabla\A(\nabla u)\big|^2\,\vrho\,dx  \leq C_0\,\int_{\Omega\cap B_{\widehat{C}_0R_0}(x_0)} f^{2}\,\vrho\,dx+\frac{C_0}{R_0^{n+2+a}}\,\left( \int_{\Omega\cap B_{\widehat{C}_0R_0}(x_0)}|\A(\nabla u)|\,\vrho\,dx \right)^2
     \end{equation}
with
     \begin{equation}
    \left(\int_{\Omega\cap B_{\widehat{C}_0R_0}(x_0)} \big|\A(\nabla u)\big|\,\vrho\,dx\right)^2 \leq \left( \int_{\Omega\cap B_{\widehat{C}_0 R_0}(x_0)}\vrho\,dx\right)^{2/p}\,\left(\int_{\Omega\cap B_{\widehat{C}_0 R_0}(x_0)}|\nabla u|^{p}\,\vrho\,dx \right)^{2/p'}\,.
  \end{equation}
\end{theorem}

\medskip

The proof of Theorems \ref{main:thmneu}, \ref{thm:loc1} and \ref{thm:loc2} exploits a suitable Reilly's identity. The argument must be implemented via a careful approximation scheme, which involves the approximation of the domain, of the operator and of the weight. 

The choice of log-concave weights is crucial both for exploiting Reilly's identity and for the validity of a Poincar\'e inequality (see \cite{fe:ni,DNP}), which is a fundamental ingredient of our approach, for instance to prove energy estimates. Poincar\'e inequality is also used to prove a compactness theorem for weighted Sobolev spaces, see Theorem \ref{thm:cpt} below, which is of independent interest.  

Our approach is quite flexible and can be used to deal with more general problems, such as anisotropic operators as the ones considered in \cite{accfm} and \cite{miao}. This would add more technical issues, making the presentation of the manuscript far too involved. For this reason, we preferred to confine our investigation to the classical $p$-Laplace operator.

\medskip

We also mention that the study of Dirichlet problems for \eqref{eq_main1} is much more delicate and a general theory seems false. As an example, one can think to the function $u(x',x_n)=x_n^{1-a}$ in $\Omega = \{x_n>0\} \cap B_1$ with weight $\vrho(x',x_n)= x_n^a$, for $a \in (0,1)$. One has that $\diver (\vrho \nabla u) = 0$ but $\nabla^2u \not\in L^2(\Omega;\vrho)$, see also Section \ref{sec:dir} for further considerations.

\medskip

The paper is organized as follows. In Section \ref{sec:prelim} we introduce some notation, provide some example and discussion on log-concave functions, and we give some preliminary results related to the approximation arguments that we will use.  In Section \ref{sec:poincare} we will give a useful Poincar\'e inequality and prove some compactness theorems.  Section \ref{sec:fundlem} contains some fundamental lemmas for vector fields which will be used to obtain our main results. In Section \ref{sec:proof} we prove Theorem \ref{main:thmneu}, while Theorems \ref{thm:loc1} and \ref{thm:loc2} are proved in Section \ref{section_thm23}. In Section \ref{sec:dir} we give some remarks on the Dirichlet problem.

\section{Preliminaries}\label{sec:prelim}

In this section we give some properties and examples of log-concave functions and set up an approximations scheme. 

\subsection{Log-concave functions} 
We consider a positive weight $\vrho:\Omega \to(0,\infty)$ which is log-concave, i.e. its logarithm is concave, or equivalently
\begin{equation}
    \vrho\big(\l x+(1-\l)y\big)\geq [\vrho(x)]^\l[\vrho(y)]^{1-\l}\,,
\end{equation}
for all $x,y\in \Omega$ and $\l\in [0,1]$.
Therefore, we may write it in the form
\begin{equation}
    \vrho(x)=e^{-h(x)}\,,
\end{equation}
for some convex function $h:\Omega\to \R$. Thanks to the properties of convex functions, it follows that $\vrho$ is locally Lipschitz continuous in $\Omega$ and it may vanish only at the boundary $\partial \Omega$. Moreover, if $\Omega$ is bounded, then $\vrho$ attains its maximum in the interior of $\Omega$.

In Theorems \ref{thm:loc1}-\ref{thm:loc2}, we shall consider the continuous extension of $\vrho$ at the boundary of $\Omega$, which we still denote by $\vrho$. Since $\vrho$ may vanish at the boundary, in this case we shall consider the corresponding extension of $h$, still denoted by $h$, as a function $h:\overline \Omega\to \R \cup \{+\infty\}$ which is still convex. Hence, the zero level set of $\vrho$, possibly contained in $\partial \Omega$, corresponds to the points where $h=+\infty$. 

\begin{example} \label{ex1}
Given $\Omega\subset \R^n$ a bounded convex domain, let $d:\Omega\to \R$ be the distance function from $\partial \Omega$
\begin{equation}
    d(x)=\mathrm{dist}(x,\partial \Omega)=\inf_{y\in \partial \Omega} |y-x| \,.
\end{equation}
It is well-known that $d$ is concave \cite[Theorem 10.1]{del:zo} and thus it is also log-concave. 
Since any positive power of a positive concave functions is log-concave, then $d^a$ is log-concave for any $a>0$.
\end{example}

\begin{example} \label{ex2}
Let $\psi: \R^{n-1} \to [0,\infty)$ be a convex function and let $\Omega = \{ x=(x',x_n) \in \R^n :\ x_n > \psi(x') \}$. Then $\vrho(x)=x_n^a$ is log-concave for any $a> 0$. 
\end{example}

\subsection{Approximation of log-concave weights}\label{sec:apprweight}

We will need to approximate the weight $\vrho$ via a sequence of smooth, log-concave functions.
To this end, we use the standard convolution kernel
\begin{equation*}
    \eta(x)=\begin{cases}
        c_n\,e^{-\frac{1}{1-|x|^2}}\quad & x\in B_1
        \\
        0\quad & x\not\in B_1\,,
    \end{cases}
\end{equation*}
where the constant $c_n$ is such that $\int_{\R^n}\eta\,dx=1$. Then, for $\e>0$, we set
\begin{equation*}
    \eta_\e(x)=\frac{1}{\e^n}\,\eta(x/\e^n)\,.
\end{equation*}
Observe that $\eta_\e$ is a log-concave function on the ball $B_\e$, and it is supported on its closure.

By exploiting the properties of convolution, we have the following approximation result.
\begin{proposition}\label{prop:conv}
Let $\vrho$ be a positive log-concave function in a bounded convex domain $\Omega$. Let $\{\Omega_m\}_{m\in \N}$ be a sequence of convex domains in $\R^n$ such that
    \begin{equation}\label{sub:omm}
        \Omega_m\subset\subset\Omega_{m+1}\subset\subset \Omega\,.
    \end{equation}
    
Then, for any fixed $m\in \N$, there exists a sequence $\{\rme\}_{\e>0}$ of log-concave functions in $\Omega_m$ such that
        \begin{equation}\label{vrho:m}
        \rme \in C^\infty(\overline{\Omega}_m)\,,\quad \inf_{\Omega_m}\rme\geq c_m>0\quad\text{ for all $\e>0$,}
    \end{equation}
    and
\begin{equation}\label{conv:vrhom}
        \rme \xrightarrow{\e\to 0}\vrho\quad\text{uniformly in $\overline{\Omega}_m$.}
    \end{equation}
\end{proposition}
Proposition \ref{prop:conv} is a consequence of the convolution properties and of the following theorem.
\begin{theorem}\label{thm:prek}
    Let $A\subset \R^n$ be a convex open set, and let  $g:\R^n\times A\to [0,\infty)$ be a log-concave function. Then the function
    \begin{equation*}
        x\mapsto \int_{A} g(x,y)\,dy
    \end{equation*}
    is log-concave in $\R^n$.
\end{theorem}
Theorem \ref{thm:prek} can be found in \cite[Theorem 6]{prekopa}, \cite[Proof of Theorem 2]{prekopa1}. The theorem is stated for log-concave functions $g(x,y)$ defined on $(x,y)\in \R^n\times \R^n$, but its proof also works for $y$ merely defined on convex domains $A$.

\begin{proof}[Proof of Proposition \ref{prop:conv}]
Let $\vrho$ be a positive log-concave function of the form $\vrho=e^{-h}$, for some convex function $h$ on $\Omega$. 
In particular, the function $h$ is convex and Lipschitz continuous in $\Omega_{m}$ for all $m\in \N$. Hence, for $m \in \N$ fixed, we consider the sets $\Omega_m\Subset\Omega_{m+1}\Subset\Omega_{m+2}\Subset \Omega$ and we consider the convex, Lipschitz extension of $h|_{\Omega_{m+2}}$ to $\R^n$, which we denote by $H$. We set $\tilde{\vrho}=e^{-H}$, so that $\tilde{\vrho}$ is positive and log-concave in $\R^n$, and it coincides with $\vrho$ in $\Omega_{m+2}$.
Now, for any $\e>0$ such that
\begin{equation*}
   0< \e<\frac{1}{2}\min\big\{\mathrm{dist}(\Omega_{m},\Omega_{m+1});\,\mathrm{dist}(\Omega_{m+1},\Omega_{m+2})\big\}\,,
\end{equation*}
we define the convolution function  
    \begin{equation}\label{convolution}
    \rme(x)=\int_{B_\e} \tilde{\vrho}(x-y)\,\eta_\e(y)\,dy=\int_{B_\e} \vrho(x-y)\,\eta_\e(y)\,dy\,,\quad x\in \Omega_{m+1}\,,
\end{equation}
where the second equality follows from the fact that $\Omega_{m+1}+B_\e\Subset \Omega_{m+2}$ by our choice of $\e$, and $\tilde{\vrho}=\vrho$ in $\Omega_{m+2}$. Here, $\Omega_{m+1}+B_\e$ denotes the Minkowsky sum between the sets $\Omega_{m+1}$ and $B_\e$, which is still a  convex set.

By the standard properties of convolution $\rme\in C^\infty(\Omega_{m+1})$,
\begin{equation*}
    \inf_{\Omega_m}\rme\geq \inf_{\Omega_m+B_\e}\vrho\geq \inf_{\Omega_{m+1}}\vrho\eqqcolon c_m>0\,,
\end{equation*}
and $\rme \xrightarrow{\e\to 0} \vrho$ uniformly in $\overline{\Omega}_m$.

Furthermore, by applying Theorem \ref{thm:prek} with $g(x,y)=\vrho(x-y)\,\eta_\e(y)$, $y\in A=B_\e$, it follows that $\rme$ is log concave in $\Omega_{m+1}$, thus completing the proof.
\end{proof}

\subsection{Stress field approximation} We will also need to approximate the vector field \begin{equation}\label{def:A}
    \A(\xi)=|\xi|^{p-2}\xi
\end{equation}
with a sequence of regularized vector field. Let $\e>0$, and define
\begin{equation}\label{def:Ae}
    \A_\e(\xi)=\left[\e^2+|\xi|^2 \right]^{\frac{p-2}{2}}\xi\,.
\end{equation}
 We have $\A_\e\in C^\infty(\R^n)$, and 
\begin{equation}\label{conv:Ae}
    \A_\e(\xi)\xrightarrow{\e\to 0} \A(\xi)\quad\text{uniformly in $\{|\xi|\leq M\}$}\,,
\end{equation}
for all $M>0$.
Furthermore, there holds
\begin{equation}\label{eigen1:1}
    \min\{p-1,1\}\,[\e^2+|\xi|^2]^{\frac{p-2}{2}}\,\mathrm{Id}\leq \nabla_\xi \A_\e(\xi)\leq  \max\{p-1,1\}\,[\e^2+|\xi|^2]^{\frac{p-2}{2}}\,\mathrm{Id}\,,
\end{equation}
for $\xi\in \R^n$, where $\mathrm{Id}$ is the $n\times n$ identity matrix. Hence, by denoting $\l^\e_{\min},\,\l^\e_{\max}$ the smallest and largest eigenvalue of the symmetric matrix $\nabla_\xi\A_\e(\xi)$, there holds
\begin{equation}\label{eigen:2}
    \frac{\l^\e_{\min}}{\l^\e_{\max}}\geq \frac{\min\{p-1,1\}}{\max\{p-1,1\}}\,.
\end{equation}

We close this section by recalling an elementary algebraic inequality which can be obtained, for instance, via \cite[Lemma 3.1]{guarnotta}.
Let $X\in \R^{n\times n}$ be a matrix of the form $X=P\,S$, with $P,S$ symmetric, $P$ positive definite, and  let $r=\l_{\min}(P)/\l_{\max}(P)$ be the ratio between the smallest and largest eigenvalue of $P$. Then
\begin{equation}\label{elem:ineq}
    \mathrm{tr}(X^2)\geq 2\,\frac{r}{1+r^2}\,|X|^2\,,\quad X=P\,S \,.
\end{equation}

\section{Poincar\'e inequality and compactness theorems}  \label{sec:poincare}
Let $\Omega$ be a bounded convex domain. For log-concave weights, there holds a weighted Poincar\'e inequality \cite{fe:ni,DNP}
\begin{equation}\label{poincare}
    \bigg(\int_\Omega |u|^p\,\vrho\,dx\bigg)^{1/p}\leq C(p)\,d_\Omega\, \bigg(\int_\Omega |\nabla u|^p\,\vrho\,dx\bigg)^{1/p}\,,
\end{equation}
for all functions $u\in W^{1,p}(\Omega;\vrho)$ such that $\int_\Omega |u|^{p-2}u\,\vrho\,dx=0$.
This last term is nonlinear in the $u$-variable. However, via standard modifications, we may rewrite Poincar\'e inequality  in the usual and more convenient form.
\begin{lemma}\label{lemm:poincare}
Let $\Omega$ be a bounded convex domain.  For each $u\in W^{1,p}(\Omega;\vrho)$, there exists a constant $c_u$ such that
    \begin{equation}\label{zeromeanu}
        \int_\Omega |u-c_u|^{p-2}(u-c_u)\,\vrho\,dx=0\,.
    \end{equation}
    Furthermore, if we denote by 
    $$(u)_{\Omega;\vrho}=\left(\int_\Omega\vrho\,dx\right)^{-1}\int_\Omega u\,\vrho\,dx$$
    the mean average of $u$ w.r.t. the measure $\vrho\,dx$, we have the following Poincar\'e inequality
    \begin{equation}\label{ppoincare}
        \int_\Omega |u-(u)_{\Omega;\vrho}|^p\,\vrho\,dx\leq C(p)\,d_\Omega^p\,\int_\Omega |\nabla u|^{p}\,\vrho\,dx\,.
    \end{equation}
\end{lemma}

\begin{proof}
   Let $u\in W^{1,p}(\Omega;\vrho)$, and consider the function
    \begin{equation*}
        \R\ni c\mapsto K(c)=\int_\Omega |u-c|^{p-2}(u-c)\,\vrho\,dx\,.
    \end{equation*}
    By dominated convergence theorem, this is a continuous function; also, the monotonicity of the function $g(t)=|t|^{p-2}t$, $t\in \R$, ensures that
    \begin{align*}
        &|u-c|^{p-2}(c-u)+|u|^{p-2}u\geq 0\,,\quad c\geq 0
        \\
        &|u-c|^{p-2}(u-c)-|u|^{p-2}u\geq 0\,,\quad c<0\,.
    \end{align*}
    Since the function $|u|^{p-2}u\in L^{p'}(\Omega;\vrho)$, it follows by monotone convergence theorem that $$\lim_{c\to + \infty} K(c)=- \infty\quad\text{and}\quad \lim_{c\to -\infty} K(c)=+\infty.$$ 
    Thus, by continuity, there exists $c_u$ satisfying $K(c_u)=0$, i.e. \eqref{zeromeanu}.
Next, by Poincar\'e inequality \eqref{poincare} applied to $u-c_u$, and Jensen inequality,  we get
\begin{align*}
    \int_\Omega |u-(u)_{\Omega;\vrho}|^p\,\vrho\,dx & \leq 2^{p-1} \int_\Omega |u-c_u|^p\,\vrho\,dx+2^{p-1}\int_\Omega |c_u-(u)_{\Omega;\vrho}|^p\,\vrho\,dx
    \\
    &\leq 2^{p-1}\,C(p)\,d_\Omega^p\,\int_\Omega |\nabla u|^p\,\vrho\,dx+2^{p-1}\int_\Omega |u-c_u|^p\,\vrho\,dx
    \\
    &\leq 2^{p}\,C(p)\,d_\Omega^p\,\int_\Omega |\nabla u|^p\,\vrho\,dx\,,
\end{align*}
    which is \eqref{ppoincare}.
\end{proof}

As a consequence of \eqref{ppoincare}, we have that for any function $u\in W^{1,2}(\Omega;\vrho)$, there holds
\begin{equation}\label{poinc:compl}
    \int_{\Omega} |u|^2\,\vrho\,dx\leq C_* \,d_\Omega^2\,\int_{\Omega}|\nabla u|^2\,\vrho\,dx+\left(\int_\Omega\vrho\,dx\right)^{-1}\,\left(\int_\Omega |u|\,\vrho\,dx \right)^2\,,
\end{equation}
where we set $C_*=C(2)$, i.e. the constant appearing in \eqref{poincare} for $p=2$. 

Thanks to the Poincar\'e inequality \eqref{ppoincare} we can prove the following compactness theorem.

\begin{theorem}[Compactness of $W^{1,p}(\Omega;\vrho)$]\label{thm:cpt}
    Let $\Omega\subset \R^n$ be a bounded convex domain, and let $\{u_j\}_{j\in \N}$ be a sequence of functions in $W^{1,p}(\Omega;\vrho)$ such that
    \begin{equation}
        \int_\Omega |u_j|^{p}\,\vrho\,dx+ \int_\Omega |\nabla u_j|^{p}\,\vrho\,dx\leq C\quad \text{for all $j\in \N$,}
    \end{equation}
   for some constant $C>0$. Then there exists a subsequence $\{u_{j_k}\}_{k\in \N}$ and a function $u\in W^{1,p}(\Omega;\vrho)$ such that
    \begin{equation}
        u_{j_k}\xrightarrow{k\to\infty} u\quad\text{in $L^p(\Omega;\vrho)$.}
    \end{equation}
\end{theorem}

\begin{proof}
In order to prove prove Theorem \ref{thm:cpt}, we are going to use the compactness theorem of Riesz in $L^p(\mu)$ \cite[Theorem 4.7.28, pag. 295]{bogachev}, and show the following, equivalent result.

\medskip 

\noindent {\it \emph{Assertion:}    Let $K\subset W^{1,p}(\Omega;\vrho)$ be a bounded subset of $W^{1,p}(\Omega;\vrho)$. Then $K$ is precompact in $L^p(\Omega;\vrho)$.}

\medskip

Let $\delta>0$ and let $\{\tilde{Q}_\delta\}$ be a partition of $\R^n$ given by dyadic cubes of diameter less or equal to $\delta$, and denote by $\{\tilde{Q}_\delta^i\}_{i=1}^N$ the subset of cubes such that $\tilde{Q}_\delta \cap \Omega \neq \emptyset$. Let  $Q^i_\delta = \tilde{Q}_\delta^i \cap \Omega$ and set $\pi=\{Q^i_\delta\}_{i=1}^N$. Therefore $Q^i_\delta$ are convex, $\mathrm{diam}(Q^i_\delta)\leq \delta$, and they satisfy
\begin{equation}\label{partition}
    \bigcup_{i=1}^N Q^i_\delta=\Omega\quad\text{$\vrho\,dx$-a.e. on $\Omega$, and}\quad Q^i_\delta\cap Q^j_\delta=\emptyset\quad i\neq j\,,
\end{equation}
    that is $\pi$ is a partition of $\Omega$ with respect to the measure $d\mu=\vrho\,dx$.
For $f\in W^{1,p}(\Omega;\vrho)$ and $\pi$ as above, we define  the operator $\mathbb{E}^\pi f$ as
\begin{equation*}
    \mathbb{E}^\pi f(x)=\sum_{i=1}^N(f)_{Q^i_\delta;\vrho}\,\chi_{Q^i_\d}(x)\quad x\in \Omega\,.
\end{equation*}
 Now set $C_K=\sup_{f\in K}\|f\|^p_{W^{1,p}(\Omega;\vrho)}$, and let $f$ be an arbitrary function belonging to $K$. 
 
 By \eqref{partition}, we have $f(x)=\sum_{i=1}^N f(x)\,\chi_{Q_\delta^i}(x)$ for $\vrho\,dx$-a.e. $x\in \Omega$, so that by using Poincar\'e inequality \eqref{ppoincare} and \eqref{partition}, we infer
 \begin{align*}
         \int_\Omega \left|f-\mathbb{E}^\pi f \right|^p\,\vrho\,dx & =\int_\Omega \sum_{i=1}^N\left| f(x)-(f)_{Q^i_\delta;\vrho} \right|^p\,\chi_{Q^i_\d}(x)\,\vrho\,dx
         \\
         &=\sum_{i=1}^N\int_{Q^i_\d}|f-(f)_{Q^i_\delta;\vrho}|^p\,\vrho\,dx
         \\
         &\leq C(p)\,\sum_{i=1}^N\mathrm{diam}(Q^i_\d)^p\int_{Q^i_\d}|\nabla f|^p\,\vrho\,dx
         \\
         &\leq C(p)\,\delta^p\,\sum_{i=1}^N\int_{Q^i_\d}|\nabla f|^p\,\vrho\,dx
         \\
         &=C(p)\,\delta^p\,\int_\Omega|\nabla f|^p\,\vrho\,dx\leq C(p)\,\delta^p\,C_K\,.
 \end{align*}
 Therefore, by letting $\d\to 0$ and by using \cite[Theorem 4.7.28, pag. 295]{bogachev}, we deduce that $K$ is precompact in $L^p(\Omega;\vrho)$, that is our thesis.
\end{proof}

Thanks to the compactness Theorem \ref{thm:cpt}, in a standard way, we can prove Poincar\'e-type inequalities for $W^{1,p}(\Omega;\vrho)$ for functions vanishing on an hypersurface $\Gamma_0\subset\{\vrho>0\}$ with non-zero $\H^{n-1}$-measure. This Poincar\'e inequality will be particularly useful in Section \ref{section_thm23} when we will deal with local estimates at the boundary for unbounded domains. In particular it will be used in an approximation argument to show convergence to the desired solution by considering a sequence of convex sets $\Omega_m\Subset \Omega$ converging in the Hausdorff sense to $\Omega$, and showing that the quantitative constants are independent on $m$.

For our purposes (and for the sake of simplicity), it will be enough to consider $\Omega$ and $\Omega_m$ to be a cylinder intersected with the epigraph of convex functions $F,F_m$, namely
\begin{equation}\label{supergraph}
    \Omega=\{x=(x',x_n):|x'|<R\,,\quad F(x')<x_n<K\}\,,
\end{equation}
 and
\begin{equation}\label{supergraph:m}
    \Omega_m=\{x=(x',x_n):|x'|<R\,,\quad F_m(x')<x_n<K\}
\end{equation}
for some constants $R,K>0$, with 
\begin{equation}\label{f:fm}
    F(x')<F_m(x')\quad |x'|\leq 2R\,,\quad F_m\xrightarrow{m\to\infty} F\quad\text{uniformly for $|x'|\leq 2R$,}
\end{equation}
and the Lipschitz constants $L_{F_m}$ of $F_m$ satisfy
\begin{equation}\label{Lip:fm}
    L_{F_m}\leq 2\,L_F\,.
\end{equation}
We notice that \eqref{Lip:fm} holds thanks to the properties of convex functions coupled with \eqref{f:fm}; actually, in our case, \eqref{Lip:fm} will follow from the fact that $F_m$ will be the  convolution of $F$, up to an additive constant. We let \begin{equation}\label{def:E}E=\{x=(x',x_n),\:|x'|< 2R,\quad F(x')<x_n<2K\}\,,
\end{equation}
and assume that
\begin{equation}\label{vrho:graph}
    \vrho\text{ is positive and log-concave in $E$, and $\vrho\in C^0(\overline{E})$.}
\end{equation}
 By considering the continuous extension of $\vrho$ to the boundary of $E$, we notice that $\vrho$ may only vanish on $\partial E$. In particular, when one considers the sets $\Omega$ and $\Omega_m$ defined in \eqref{supergraph} and \eqref{supergraph:m}, respectively, one obtains that $\vrho$ may vanish only at points $(x',F(x'))$ on the graph of $F$, i.e. in a portion of $\partial \Omega$, and $\vrho >0 $ on $\overline \Omega_m$ for any $m \in \N$.

\begin{theorem}\label{thm:zerotrpo}
    Let $\Omega,\Omega_m$ be satisfy \eqref{supergraph}-\eqref{Lip:fm} and let $\vrho$ satisfy \eqref{vrho:graph}. Then there exists a constant $C=C(n,p,\Omega,\vrho)$ such that
    \begin{equation}\label{poinc:zero}
        \int_{\Omega_m}|w|^p\,\vrho\,dx\leq C\int_{\Omega_m}|\nabla w|^p\,\vrho\,dx
    \end{equation}
    for all functions $w\in W^{1,p}(\Omega_m;\vrho)$ satisfying
    \begin{equation}\label{zero:trace}
        w(x',K)=0\quad\text{for all $|x'|<R$.}
    \end{equation}
\end{theorem}
A few comments are in order. Thanks to assumption \eqref{vrho:graph}, $\{ \vrho = 0\} \subset \{x_n = F(x')\}$ and then we have $W^{1,p}(\Omega_m;\vrho)=W^{1,p}(\Omega_m)$, hence traces are well defined. Theorem \ref{thm:zerotrpo} states that if $u\in W^{1,p}(\Omega_m)$ has zero trace on $\{x=(x',K),\,|x'|<R\}$ then it satisfies the weighted Poincar\'e inequality \eqref{poinc:zero}, with constant $C$ independent of $m$. 

Before proving Theorem \ref{thm:zerotrpo}, we need the following lemma.
\begin{lemma}\label{lem:tech}
    Let $\Omega,\Omega_m,\vrho$ satisfy \eqref{supergraph}-\eqref{vrho:graph}, and let $\Phi_m$ be the function defined as
    \begin{equation}\label{def:Phim}
        \Phi_m(x)=\Phi_m(x',x_n)=
        \begin{cases}
            x\quad & \text{if $x\in \Omega_m$}
            \\
            (x',2\,F_m(x')-x_n)\quad & \text{if $x\in \Omega\setminus \Omega_m$.}
        \end{cases}
    \end{equation}
Then there exists $m_0>0$ such that
\begin{equation}\label{prop:vrhom}
    \vrho(x)\leq 2\,\vrho \big( \Phi_m (x)\big)\,,\quad x\in \Omega,
\end{equation}
    for all $m>m_0$.
\end{lemma}

The function $\Phi_m$ is nothing else than the  ``even''  reflection of a point $x$ with respect to the graph of $F_m$, as depicted in figure.  Also, $\Phi_m$ is Lipschitz continuous, and it is the identity map when restricted to $\mathrm{Graph}\, F_m$.

\begin{figure}[ht]\label{fig0}
\centering
\includegraphics[width=0.8\textwidth]{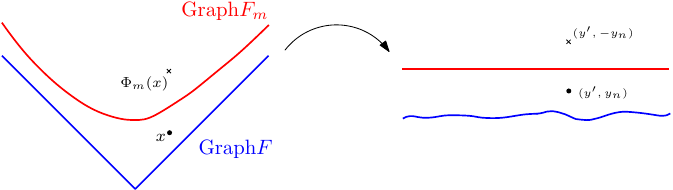}
\caption{}
\end{figure}

Roughly speaking, the point $x$ is closer to $\mathrm{Graph}\,F$ than $\Phi_m(x)$ is, so that inequality \eqref{prop:vrhom} follows by a continuity-convexity argument coupled with the Hausdorff convergence \eqref{f:fm}.

\begin{proof}[Proof of Lemma \ref{lem:tech}]
    First observe that, owing to \eqref{f:fm}, we have that $\Phi_m(\Omega\setminus \Omega_m)\subset \Omega_m$ for $m$ large enough, and then the function $\vrho\circ \Phi_m$ is well defined. 
    Assume by contradiction that \eqref{prop:vrhom} is false. Then we may find a sequence $\{x_m\}_{m\in \N}\subset \Omega\setminus \Omega_m$ such that
    \begin{equation}\label{contradi}
        \vrho(x_m) > 2\,\vrho\big( \Phi_m(x_m)\big)\,,\quad \text{for all $m\in \N$.}
    \end{equation}
    By compactness, there exists a subsequence, still denoted by $x_m$, such that $x_m\to x_0$ for some point $x_0 \in \overline{\Omega}$. By the convergence property \eqref{f:fm}, the continuity of $\vrho$, the definition \eqref{def:Phim} of the map
    $\Phi_m$ and \eqref{contradi}, we necessarily have that $x_0$ belongs to $\mathrm{Graph}\,F$, namely $x_0=(x_0',F(x'_0))$ for some $x_0'\in \R^{n-1}$, $|x'_0|\leq R$.
    
    We write $\vrho=e^{-h}$, with $h$ convex, defined in $\{(x',x_n):\,|x'|\leq 2R,\, F(x')<x_n\leq 2K\}$, and we  distinguish two cases.
    \begin{enumerate}[label=(\roman*)]
        \item $\vrho(x_0)>0$. Since $\Phi_m(x_m)\to (x'_0,F(x'_0))=x_0$ thanks to \eqref{f:fm}, we have that both $\vrho(x_m)$ and $\vrho(\Phi_m(x_m))$ converge to $\vrho(x_0)>0$. By letting $m\to\infty$ in \eqref{contradi} we find $\vrho(x_0)\geq 2\,\vrho(x_0)>0$, a contradiction.
        
        \item $\vrho(x_0)=0$. It follows that $h(x_0)=+\infty$. Since $h(x_m)\to h(x_0)=+\infty$, we may find $m_0$ such that 
        \begin{equation}\label{contradi2}
            h(x_m)>\sup_{x=(x',K)}h(x)\quad\text{for all $m\geq m_0$.}
        \end{equation}
 For any $m \geq m_0$, we write $x_m=(x'_m,(x_m)_n)$, and consider the line segment $\ell_m=\{(x'_m,z): (x_m)_n<z<K\}$. By convexity, the function $h$ restricted to $\ell_m$ admits a maximum either at $(x'_m,K)$, or at $(x'_m,(x'_m)_n)=x_m$, hence it has a unique maximum at $x_m$ by \eqref{contradi2}. Since by definition \eqref{def:Phim}, we have that $\Phi_m(x_m)\in \ell_m$, it follows that $h(x_m)\geq h\big(\Phi_m(x_m) \big)$, hence $\vrho(x_m)\leq \vrho\big( \Phi_m(x_m)\big)$, thus contradicting \eqref{contradi}.
    \end{enumerate}
\end{proof}
We are now in the position to prove Theorem \ref{thm:zerotrpo}.

\begin{proof}[Proof of Theorem \ref{thm:zerotrpo}]
    By contradiction, assume that there exists a sequence of functions $w_m\in W^{1,p}(\Omega_m;\vrho)=W^{1,p}(\Omega_m)$ such that
    \begin{equation*}
        \int_{\Omega_m} |w_m|^p\,\vrho\,dx\geq m\,\int_{\Omega_m} |\nabla w_m|^p\,\vrho\,dx\,,\quad\text{for all $m\in \N$,}
    \end{equation*}
    and $w_m(x',K)=0$ for all $|x'|<R$. Then, setting $v_m=w_m/\|w_m\|_{L^p(\Omega_m;\vrho)}$, we have that
    \begin{equation}\label{contradi:3}
        \int_{\Omega_m} |v_m|^p\,\vrho\,dx=1\,,\quad \int_{\Omega_m}|\nabla v_m|^p\,\vrho\,dx\leq \frac{1}{m}\,,
    \end{equation}
    and $v_m\in W^{1,p}(\Omega_m)$, as well as  $v_m(x',K)=0$ for all $m\in \N$. Now we define \begin{equation*}
        \tilde{v}_m(x)=v_m\big( \Phi_m(x)\big),\quad x\in \Omega\,.
    \end{equation*}
    Since $\Phi_m$ is Lipschitz continuous, and it coincides with the identity map on the boundary of $\Omega_m$ (so it preserves the trace on $\partial \Omega_m$), we have that $\tilde{v}_m\in W^{1,p}(\Omega)$-- see also the discussion in \cite[Formula (13.4)]{leo}. Also, recalling that $\Phi_m(\Omega\setminus \Omega_m)\subset \Omega_m$ for $m$ large, and since its Jacobian $|\mathrm{det}(\nabla \Phi_m)|=1$, by the chain rule, \eqref{Lip:fm} and the change of variables $y=\Phi_m(x)$, we find 
    \begin{align}\label{rem:util1}
        \int_{\Omega\setminus \Omega_m} |\nabla \tilde{v}_m(x)|^p\,\vrho\big(\Phi_m(x) \big)\,dx & \leq C(n,L_F)\,\int_{\Omega\setminus \Omega_m}\big|\nabla v_m\big(\Phi_m(x) \big)\big|^p \,\vrho\big(\Phi_m(x) \big)\,dx
        \\
        &\leq C'\,\int_{\Omega_m}|\nabla v_m(y)|^p\,\vrho(y)\,dy\leq \frac{C''}{m}\,,\nonumber
    \end{align}
    where the last inequality is due to \eqref{contradi:3}.
    Coupling the latter inequality with \eqref{prop:vrhom} and \eqref{contradi:3}, we deduce that
    \begin{equation}\label{contradi:4}
        \int_{\Omega} |\nabla \tilde{v}_m|^p\,\vrho\,dx\leq \frac{2C+1}{m}\,,\quad\text{for all $m\in \N$.}
    \end{equation}
 Moreover, a similar argument shows that
    \begin{equation}
        \int_\Omega |\tilde{v}_m|^p\,\vrho\,dx\leq C\,,\quad\text{for all $m\in \N$.}
    \end{equation}
 Therefore, we are in the position to apply the compactness Theorem \ref{thm:cpt} and extract a subsequence $\tilde{v}_{m_k}$ such that
 \begin{equation*}
     \tilde{v}_{m_k}\xrightarrow{k\to \infty} v\quad\text{weakly in $W^{1,p}(\Omega;\vrho)$, strongly in $L^p(\Omega;\vrho)$.}
 \end{equation*}
 By the lower-semicontinuity of the norm and \eqref{contradi:4}, we have that $\int_\Omega |\nabla v|^p\,\vrho\,dx=0$, thus $\nabla v=0$  a.e. in $\Omega$ and $v$ is constant in $\Omega$. Since $\vrho>0$ in $\Omega$, the trace $v(x',K)$ is well defined, and furthermore $\tilde{v}_{m_k}(x',K)\xrightarrow{k\to\infty} v(x',K)$ for a.e. $x'$ by the classical trace theory of Sobolev spaces.
 
Being $\tilde{v}_{m_k}(x',K)\equiv 0$ for all $k\in \N$,  we deduce that $v=0$ in $\Omega$.
 On the other hand, by the strong $L^p(\Omega;\vrho)$ convergence and \eqref{contradi:3}, we have
 \begin{equation*}
     \int_\Omega |v|^p\,\vrho\,dx=\lim_{k\to\infty}\int_\Omega |\tilde{v}_{m_k}|^p\,\vrho\,dx\geq \lim_{k\to\infty}\int_{\Omega_{m_k}} |\tilde{v}_{m_k}|^p\,\vrho\,dx= 1\,,
 \end{equation*}
 which contradicts the fact that $v\equiv 0$ on $\Omega$, hence the thesis.
\end{proof}

\begin{remark}\label{remark:utile}
    \rm{The proof of Theorem \ref{thm:zerotrpo} shows that, given a function $v\in W^{1,p}(\Omega_m)$, it is possible to construct its extension 
 $\tilde{v}(x)=v\big( \Phi_m(x)\big)$ for $x\in \Omega$, which satisfies $v\in W^{1,p}(\Omega)$, and
    \begin{equation}
        \int_\Omega |\nabla \tilde{v}|^p\,\vrho\,dx\leq C\,\int_{\Omega_m} |\nabla v|\,\vrho\,dx\,,\quad  \int_\Omega | \tilde{v}|^p\,\vrho\,dx\leq C\,\int_{\Omega_m} |v|^p\,\vrho\,dx\,,
    \end{equation}
    for a constant $C>0$, only depending on $n,L_F$.}
\end{remark}

Next, in the study of local estimate for $p>2$, we will need the following Poincar\'e inequality on  ``rectangular annuli'', in the same spirit of \cite[Formula (5.4)]{cia}.
In our case, we need to work on suitable rectangles instead of cubes to deal with the graph of $F$. For $a,b>0$, we write
\begin{equation}\label{def:rectangles}
    \mathcal{R}^F_{a,b}=Q'_a\times \big(-\sqrt{n}(1+L_F)\,b,\,\,\sqrt{n}(1+L_F)\,b\big),\quad \mathcal{R}^F_a\equiv \mathcal{R}^F_{a,a}\,,
\end{equation}
where $Q'_a$ is the $(n-1)$-dimensional cube of wedge $2a$, i.e., $Q'_a=\{x'\in \R^{n-1}: \max\limits_{1\leq i\leq n-1}|x'_i|<a\} $.

\begin{lemma}\label{annuli:lemma}
    Let $\Omega$ satisfy \eqref{supergraph} with $F(0')=0$ and $K>2\sqrt{n}(1+L_F)\,R$, with $K+L_F\,R<1$. Assume in addition that
    \begin{equation}\label{a:conc}
        \vrho(x)=[g(x)]^a
    \end{equation}
    for some concave function $g:E\to (0,\infty)$, and a positive number $a>0$, where $E$ is given by \eqref{def:E}.
    
    Let $\tau$ and $\sigma$ be such that  
    $$
    \frac{R}{8\,\sqrt{n}(1+L_F)}<\sigma<\tau<\frac{R}{4\,\sqrt{n}(1+L_F)} \,.
    $$ 
    Then there exists a constant $C=C(n,\vrho,L_F)$ such that
    \begin{equation}\label{poinc:annuli}
        \int_{\Omega\cap (\mathcal{R}^F_\tau\setminus \mathcal{R}^F_\sigma)} v^2\,\vrho\,dx\leq C\,\delta^2\,\int_{\Omega\cap (\mathcal{R}^F_\tau\setminus \mathcal{R}^F_\sigma)}|\nabla v|^2\,\vrho\,dx+\frac{C}{\delta^{(n+a)}}\,\left(\int_{\Omega\cap (\mathcal{R}^F_\tau\setminus \mathcal{R}^F_\sigma)}|v|\,\vrho\,dx \right)^{2}\,,
    \end{equation}
for any $v\in W^{1,2}(\Omega;\vrho)$ and $0<\delta< \frac{R}{8\,\sqrt{n}(1+L_F)}$.
\end{lemma}

\begin{proof}
   For the moment, we assume that  $$
       \delta=\frac{\tau-\sigma}{M},$$
for some integer $M>1$. Since $F(0')=0$, the Lipschitz continuity of $F$ ensures that \begin{equation}\label{stima:fff}
  |F(x')|\leq L_F\,|x'|\leq \sqrt{n}\,L_F\,\|x'\|_\infty\,,
  \end{equation}
  where $\|x'\|_\infty=\max\limits_{j=1,\dots n-1} |x'_j|$ is the sup-norm of $x'\in \R^{n-1}$. Hence by definition \eqref{def:rectangles} of $\mathcal{R}^F_a$ we have that $\Omega\cap (\mathcal{R}^F_\tau\setminus \mathcal{R}^F_\sigma)$ is not empty for any $\sigma$ and $\tau$ as above.
We divide $Q'_\tau\setminus Q'_\sigma$ into $M^{n-1}$ dyadic cubes of wedge $\delta/2$, and we denote by $\{(x')^i\}_{i=1}^{M^{n-1}}$ their centers.  Then consider the rectangles
\begin{equation*}
    \widetilde{\mathcal{R}}^f_i\coloneqq \big((x')^i, F\big((x')^i \big)\big)+\mathcal{R}^F_{\d/2}\,,\quad i=1,\dots, M^{n-1}\,,
\end{equation*}
 which are centered at points $\big((x')^i, F\big((x')^i \big)\big)$ lying on graph of $F$. By \eqref{stima:fff} and \eqref{def:rectangles}, these rectangles provide an open partition of the graph of $F$ in $\mathcal{R}^F_\tau\setminus \mathcal{R}^F_\sigma$-- up to a set of negligible measure.
Then, for every $i$, we set
\begin{equation*}
    r_i=\frac{\sqrt{n}(1+L_F)\,\tau-\big(F\big((x')^i\big)+\d\,\sqrt{n}(1+L_F)\big)}{M}
\end{equation*}
which is positive thanks to \eqref{stima:fff}, and we divide the line segment $\{(x')^i\}\times \big[F\big((x')^i\big)+\d\,\sqrt{n}(1+L_F), \sqrt{n}(1+L_F)\,\tau \big]$ into $M$ line segments of equal length $r_i$, and centers $\{(x_n^{(i)})^j\}_{j=1,\dots M}$. Hence, if we set
\begin{equation*}
    \hat{\mathcal{R}}^F_{i,j}\coloneqq \big((x')^i,(x_n^{(i)})^j\big)+\mathcal{R}^F_{\d/2,r_i/2}\,,\quad i=1,\dots, M^{n-1}, j=1,\dots, M,
\end{equation*}
we have that, by construction, the rectangles $\{\widetilde{\mathcal{R}}^F_i,\hat{\mathcal{R}}^F_{i,j}\}_{i,j}$ form a partition of $\big(\Omega\cap (\mathcal{R}^F_\tau\setminus \mathcal{R}^F_\sigma)\big)\setminus \big(Q'_\sigma\times (\sigma,\tau)\big)$ up to a Lebesgue-negligible set.

We are left to cover the set $Q'_\sigma\times (\sigma,\tau)$. To this end, let $M_1>1$ be an integer such that 
$$\frac{\sigma}{M_1+1} \leq\delta \leq \frac{\sigma}{M_1},$$
and  consider a partition  of $Q'_\sigma\times (\s,\tau)$ made of rectangles $\{Q_k\}_{k=1,\dots, M_1^{n-1}\,M}$ of size $Q_\d$ up to a translation. 
It follows that, up to a Lebesgue negligible set, the rectangles $\{\mathcal{R}_l\}_{l=1,\dots N}=\{\widetilde{\mathcal{R}}^F_i,\hat{\mathcal{R}}^F_{i,j},Q_k \}_{i,j,k}$ form a partition of $\Omega\cap (\mathcal{R}^F_\tau\setminus \mathcal{R}^F_\sigma) $, as depicted in Figure 2. 

\begin{figure}[ht]\label{fig1}
\centering
\includegraphics[width=0.6\textwidth]{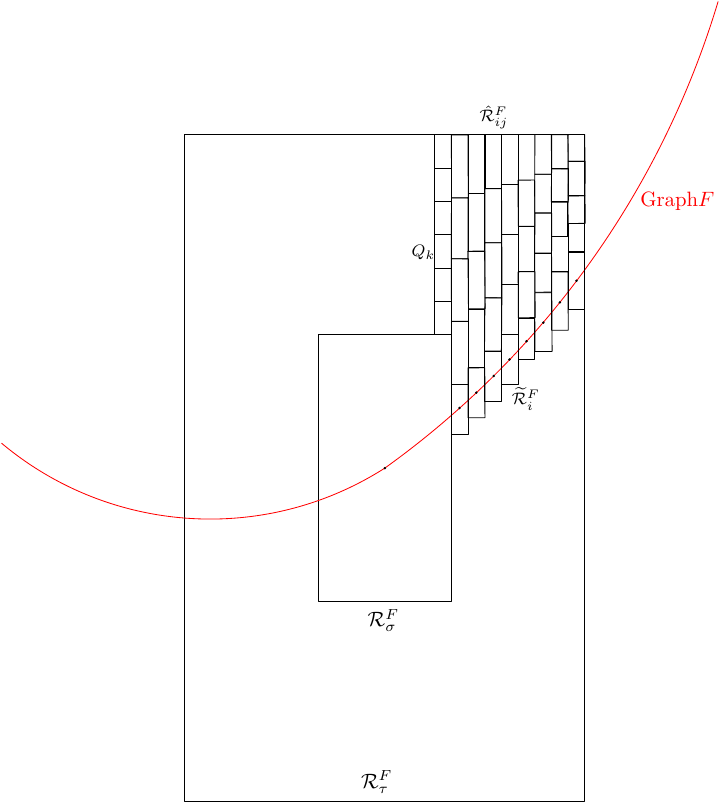}
\caption{}
\end{figure}

Observe also that
\begin{equation}\label{diam:rect}
    \mathrm{diam}\,\mathcal{R}_i\leq C(n)\,(1+L_F)\,\d\,.
\end{equation}

Now set $$c_\vrho\coloneqq \inf\limits_{x=(x',K)}\vrho(x)>0\,,$$
where $K$ is the constant in \eqref{supergraph}.

Let us define
\begin{equation}
    \Psi_F(x)=(x',x_n-F(x')),\quad \Psi_F^{-1}(y)=(y',y_n+F(y_n))\,,
\end{equation}
the map which flattens the graph of $F$ and its inverse, respectively.
We claim that
\begin{equation}\label{contained}
   \Psi_F\big((x')^i,F((x')^i) \big)+Q'_{\d/8}\times (0,\,\d/8)\subset \Psi_F\left(\widetilde{\mathcal{R}}^F_i\cap \Omega \right)\,.
\end{equation}
Indeed, for $y\in Q'_{\d/8}\times (0,\,\d\,/8)$, recalling \eqref{stima:fff}, we have
\begin{equation*}
    |y_n+F\big((x')^i+y' \big)-F\big((x')^i\big)|\leq |y_n|+L_F\,|y'|\leq \sqrt{n}\,(1+L_F)\,\d/4\,, 
\end{equation*}
and since $\Psi_F\big((x')^i,F((x')^i) \big)=\big((x')^i,0 \big)$, this implies \eqref{contained} by definition of $\widetilde{\mathcal{R}}^F_i$.
\\

We now claim that
\begin{equation}\label{stima:basso}
    \int_{\mathcal{R}_i\cap \Omega}\vrho\,dx\geq c(L_F,\vrho)\,\d^{n+a}\,,\quad\text{for all $i$,}
\end{equation}
and to prove it, we distinguish two cases.
\begin{enumerate}

        \item $\mathcal{R}_i$ is an interior rectangle, i.e., it is equal to $\hat{\mathcal{R}}^F_{i,j}$ or $Q_k$. Let $x'\in \Pi(\mathcal{R}_i)$, where $\Pi:\R^n\to \R^{n-1}$, $\Pi(x',x_n)=x'$ is the projection map, and let  
        $z_n=\inf\{t:(x',t)\in \mathcal{R}_i\}$. Consider the line segment $\ell_{x'}=\{x'\}\times [z_n,K]$. 

        By definition of $c_\vrho$, the concavity of $g$ and its nonnegativity, it follows that the graph of $g(x',\cdot)$ lies above the line segment connecting $\big(K,c_{\vrho}^{1/a} \big)$ and $(z_n,0)$, namely $g(x',x_n)\geq c_\vrho^{1/a}\frac{(x_n-z_n)}{K-z_n}$.
        In particular, being $z_n\geq F(x')$ by construction of $\mathcal{R}_i$ and  by \eqref{stima:fff}, this implies that
        \begin{align*}\vrho(x',x_n)&\geq \frac{c_\vrho}{(K-z_n)^a}(x_n-z_n)^a\geq  \frac{c_\vrho}{(K-F(x'))^a}(x_n-z_n)^a
            \\
            &\geq  \frac{c_\vrho}{(K+L_F\,|x'|)^a}(x_n-z_n)^a \geq \frac{c_\vrho\,2^a}{(K+L_F\,R)^a}(x_n-z_n)^a \\
            & \geq  c_\vrho\,2^a (x_n-z_n)^a \,,\quad \text{for all $x'\in \Pi(\mathcal{R}_i)$,}
        \end{align*}
        where in the last inequality we used that $K+L_F\,R<1$.
        
        Thus, recalling the construction of $\mathcal{R}_i$, we deduce
        \begin{equation*}
            \int_{\mathcal{R}_i}\vrho\,dx\geq c\,\int_{Q'_\d}dx'\int_{z_n}^{z_n+\d}(x_n-z_n)^a\,dx_n\geq c\,\d^{n+a}\,.
        \end{equation*}
        \item $\mathcal{R}_i=\hat{\mathcal{R}}^F_i$ is a rectangle intersecting the graph of $F$. 
        In this case, consider the function $g\circ\Psi^{-1}_F(y',y_n)=g(y',y_n+F(y'))$. For every fixed $y'$, this is a concave function of the $y_n$-variable. Hence, we may repeat the same argument as in the previous point with $z_n=0$, and considering the line segment $\{y'\}\times (0,K-F(y'))$, thus deducing that  $\vrho\circ\Psi^{-1}_F(y',y_n)\geq c\,y_n^a$. Recalling \eqref{contained}, we infer
        \begin{equation*}
            \int_{\mathcal{R}_i\cap \Omega}\vrho\,dx=\int_{\Psi_F(\mathcal{R}_i\cap \Omega)}\vrho\circ \Psi^{-1}_F\,dy\geq c\,\d^{n+a}\,,
        \end{equation*}   
        and \eqref{stima:basso} is proven.
    \end{enumerate}

Now let $v\in W^{1,2}(\Omega;\vrho)$. By Poincar\'e inequality \eqref{poinc:compl}, \eqref{diam:rect} and \eqref{stima:basso} we have 
\begin{align*}
        \int_{\Omega\cap (\mathcal{R}^F_\tau\setminus \mathcal{R}^F_\sigma)} & v^2\,\vrho\,dx=\sum_{i=1}^N \int_{\Omega\cap \mathcal{R}_i}v^2\,\vrho\,dx
        \\
        & \leq C\,\d^2\,\sum_{i=1}^N \int_{\Omega\cap \mathcal{R}_i} |\nabla v|^2\,\vrho\,dx+C\,\sum_{i=1}^N \,\left(\int_{\Omega\cap \mathcal{R}_i}\vrho\,dx \right)^{-1}\left( \int_{\Omega\cap \mathcal{R}_i}|v|\,\vrho\,dx\right)^{2}
        \\
        &\leq C\,\d^2\,\int_{\Omega\cap (\mathcal{R}^F_\tau\setminus \mathcal{R}^F_\sigma)} |\nabla v|^2\,\vrho\,dx+C\,\d^{-(n+a)}\,\left(\sum_{i=1}^N \int_{\Omega\cap \mathcal{R}_i}|v|\,\vrho\,dx\right)^{2}
        \\
        &= C\,\d^2\,\int_{\Omega\cap (\mathcal{R}^F_\tau\setminus \mathcal{R}^F_\sigma)} |\nabla v|^2\,\vrho\,dx+\frac{C}{\delta^{(n+a)}}\,\left( \int_{\Omega\cap (\mathcal{R}^F_\tau\setminus \mathcal{R}^F_\sigma)}|v|\,\vrho\,dx\right)^{2}\,.
\end{align*}
 The proof is thus completed, save that we assumed that $\d=(\tau-\sigma)/M$ for some integer $M>1$. However, for a general $\delta>0$, it suffices to take the integer $M>1$ satisfying $\frac{\tau-\s}{M+1}<\delta\leq \frac{\tau-\sigma}{M}$.
\end{proof}

\begin{remark}\label{remark:annuli}
    We exploited the particular form \eqref{a:conc} of $\vrho$ to obtain the estimate \eqref{stima:basso}, from which we deduced \eqref{poinc:annuli}. For a general log-concave function $\vrho$, estimate \eqref{stima:basso} does not hold. Consider for instance $\Omega=\R^n_+=\{x=(x',x_n)\in \R^n:x_n>0 \}$ and $\vrho(x)=\exp\{-1/x_n\}$. Since $\vrho(x)\leq x_n^{\b}$ for all $\beta>0$, an estimate of the kind  
    \begin{equation*}
        \int_{Q'_\d\times (0,\d)} \exp(-1/x_n)dx\geq c\, \d^{n+a}\,,\quad \text{for all $\d>0$}
    \end{equation*}
    cannot hold true for any $a>0$.
\end{remark}

\section{Fundamental lemmas for vector fields and generalized Reilly's identity}\label{sec:fundlem}
This section is devoted to a few differential and integral identities, which will be crucial in our proofs. The main goal is to prove Propositions \ref{prfund:neu} and \ref{prop:funddir} below, which are essentially generalizations of Reilly's identity which are suitable for the problem under consideration. 

In what follows, given a vector field $V:\Omega\to \R^n$, with $V=(V^1,\dots,V^n)$, we write
\begin{equation*}
    \nabla V=(\partial_{j} V^i)_{i,j=1,\ldots,n}\,.
\end{equation*}
We adopt the convention about summation over repeated indices.  Hence $\nabla V\,V$ is the vector whose $i$-th component agrees with $V^j\partial_{j}V^i$.

\begin{lemma}\label{lemma:1}
    Let $\Omega\subset \R^n$ be an open set. Assume $h:\Omega\to \R$, $\vrho=e^{-h}$ and  $V:\Omega\to\R^n$ are of class $C^2$ in $\Omega$. Then 
    \begin{equation}\label{id:fund}
        \vrho^{-1}\big[ \mathrm{div}\big(\vrho\,V \big)\big]^2=\vrho\,\mathrm{tr}\big[(\nabla V)^2\big]+\mathrm{div}\big(V\,\mathrm{div}(\vrho\,V)-\vrho\,\nabla V\,V \big)+\vrho\,\nabla^2 h\,V\cdot V\,.
    \end{equation}
\end{lemma}

\begin{proof}
By Schwarz Theorem on second derivatives, we have
    \begin{equation*}
        V^j\,\partial_{ij}V^i=V^i\,\partial_{ji} V^j\,.
    \end{equation*}
So, recalling that $\nabla \vrho=-\vrho\,\nabla h$, we compute
    \begin{align}\label{a0}
        \vrho^{-1}\mathrm{div}\big(\vrho\,\nabla V\,V \big)=\vrho^{-1}\partial_i\big(\vrho\,V^j\,\partial_j V^i  \big) &=\partial_i V^j\,\partial_j V^i+V^j\,\partial_{ji} V^i+\frac{\partial_i\vrho}{\vrho}\,V^j\,\partial_j V^i
        \\
        &=\mathrm{tr}\big[(\nabla V)^2 \big]+V^j\,\partial_{ij}V^i-\nabla V\,V\cdot \nabla h\,.\nonumber
    \end{align}
Let us now set
\begin{equation*}
    \mathcal{L}_h\coloneqq \vrho^{-1}\,\mathrm{div}(\vrho\,V)=e^h\,\mathrm{div}\big(e^{-h}\,V\big)\,.
\end{equation*}
so that $\mathcal{L}_h=\mathrm{div} V-\nabla h\cdot V$. Thus, we compute
\begin{align}\label{a1}
        \vrho^{-1}\mathrm{div}\big(V\,\mathrm{div}(\vrho\,V)\big)&=e^h\,\mathrm{div}\big( e^{-h}\,V\,\mathcal{L}_h\big)= \mathcal{L}_h\,\mathrm{div}V+e^h\,V\,\cdot \nabla(e^{-h}\,\mathcal{L}_h)
        \\
        &=\mathcal{L}_h\,\mathrm{div} V-(V\cdot \nabla h)\,\mathcal{L}_h+V\cdot \nabla\mathcal{L}_h=(\mathcal{L}_h)^2+V\cdot\nabla\mathcal{L}_h\nonumber
        \\
        &=(\mathcal{L}_h)^2+V\cdot\nabla\big(\mathrm{div} V-\nabla h\cdot V \big)\nonumber
        \\
        &=(\mathcal{L}_h)^2+V^i\,\partial_{ij}V^j-\nabla V\,V\cdot \nabla h-\nabla^2 h\,V\cdot V\,.\nonumber
\end{align}
By subtracting equations \eqref{a1} and \eqref{a0}, and multiplying both sides by $\vrho$, we obtain
\begin{equation*}
    \mathrm{div}\big(V\,\mathrm{div}(\vrho\,V)-\vrho\,\nabla V\,V \big)=\vrho\,(\mathcal{L}_h)^2-\vrho\,\mathrm{tr}\big[(\nabla V)^2 \big]-\vrho\,\nabla^2 h\,V\cdot V\,,
\end{equation*}
which is \eqref{id:fund}.
\end{proof}

Now assume that $\Omega$ is an open set in $\rn$ such that $\partial \Omega \in C^1$.  We 
denote by $\nu=\nu(x)$ the outward unit normal to $\Omega$ at a point $x\in \partial \Omega$.  Given a vector field $V: \partial \Omega \to \rn$, its tangential component $V_T$ is defined by 
\begin{equation*}
    V_T=V-(V\cdot \nu)\,\nu.
\end{equation*}
The notations $\nabla_T$ and $\mathrm{div}_T$ are adopted for the tangential gradient and divergence operators on $\partial \Omega$. 
Therefore, if $u\in C^1(\overline{\Omega})$, then
\begin{equation*}
    \nabla_T u=\nabla u- (\partial_\nu u)\,\nu \quad \text{on $\partial \Omega$,}
\end{equation*}
where $\partial_\nu u$ denotes the normal derivative of $u$.
Moreover, if $V\in C^1(\overline{\Omega})$,  then
\begin{equation*}
    \mathrm{div}_T V=\mathrm{div}V-(\partial_\nu V\cdot\nu) \quad  \text{on $\partial \Omega$.}
\end{equation*}
We denote by $\B$ the second fundamental form on $\partial \Omega$. This coincides with the
the shape operator (also called Weingarten operator) on  $\partial \Omega$, which agrees with $\nabla_T \nu$. Namely we have
\begin{equation*}
    \B(\zeta,\gamma)=\nabla_T\nu\,\zeta\cdot\gamma\,,\quad \ \textmd{ for any }\, \zeta,\gamma\in \nu^{\perp}\,,
\end{equation*}
where $\nu^{\perp}$ is the tangent space on $\partial \Omega$.

The next result is an integral lemma involving vector fields whose normal components vanish on the boundary.
\begin{proposition}\label{prfund:neu}
    Let $\Omega$ be a domain of class $C^2$, and let $V\in C^2(\overline{\Omega};\R^n)$, $h\in C^2(\overline{\Omega})$ and $\vrho=e^{-h}$. 
    Assume that
    \begin{equation}\label{Vnuzero}
        V\cdot \nu=0\quad\text{on $\partial \Omega$,}
    \end{equation}
    with $\nu$ the outward normal to $\partial \Omega$. Then, for every $\phi\in C^\infty_c(\R^n)$ we have
    \begin{align}\label{id:fundneu00}
        \int_\Omega \vrho^{-1}\big[ \mathrm{div}\big(\vrho\,V \big)\big]^2\,\phi\,dx= & \int_\Omega\vrho\,\mathrm{tr}\big[(\nabla V)^2\big]\,\phi\,dx+\int_{\partial \Omega}\vrho\,\B(V_T,V_T)\,\phi\,d\H^{n-1}+\int_\Omega\vrho\,\nabla^2 h\,V\cdot V\,\phi\,dx
        \\
        &-\int_\Omega \Big\{(V\cdot \nabla \phi)\,\mathrm{div}\big(\vrho\,V\big)-\vrho\,\nabla V\,V\cdot\nabla \phi\,\Big\}dx.\nonumber
    \end{align}
  In particular  
    \begin{align}\label{id:fundneu}
        \int_\Omega \vrho^{-1}\big[ \mathrm{div}\big(\vrho\,V \big)\big]^2\,dx=\int_\Omega\vrho\,\mathrm{tr}\big[(\nabla V)^2\big]\,dx+\int_{\partial \Omega}\vrho\,\B(V_T,V_T)\,d\H^{n-1}+\int_\Omega\vrho\,\nabla^2 h\,V\cdot V\,dx\,.
    \end{align}
\end{proposition}

\begin{proof}
    Under our hypothesis, the function $h$ admits a maximum, hence $\vrho=e^{-h}$ is strictly positive in $\overline{\Omega}$. We multiply  identity \eqref{id:fund} with $\phi$, and integrating and
 using the divergence theorem gives
    \begin{align*}
        \int_\Omega  \vrho^{-1}\big[ \mathrm{div}\big(\vrho\,V \big)\big]^2\,\phi\,dx= &    \int_\Omega\vrho\,\mathrm{tr}\big[(\nabla V)^2\big]\,\phi\,dx+\int_{\partial \Omega}\Big\{(V\cdot\nu)\,\mathrm{div}(\vrho\,V)-\vrho\,\nabla V\,V \cdot\nu\Big\}\,\,\phi\,d\H^{n-1}
        \\
        &+\int_\Omega\vrho\,\nabla^2 h\,V\cdot V\,\phi\,dx-\int_\Omega \Big\{(V\cdot \nabla \phi)\,\mathrm{div}\big(\vrho\,V\big)-\vrho\,\nabla V\,V\cdot\nabla \phi\,\Big\}dx
        \\
        = &    \int_\Omega\vrho\,\mathrm{tr}\big[(\nabla V)^2\big]\,\phi\,dx-\int_{\partial \Omega}\vrho\,\nabla V\,V \cdot\nu\,\phi\,d\H^{n-1}+\int_\Omega\vrho\,\nabla^2 h\,V\cdot V\,\phi\,dx
        \\
        &-\int_\Omega \Big\{(V\cdot \nabla \phi)\,\mathrm{div}\big(\vrho\,V\big)-\vrho\,\nabla V\,V\cdot\nabla \phi\,\Big\}dx\,,
    \end{align*}
    where in the second equality we used \eqref{Vnuzero}. Now, owing to \eqref{Vnuzero}, we have that $\nabla V\,V=\nabla_T V\,V_T+\partial_\nu V\,(V\cdot\nu)=\nabla_T V\,V_T$. It follows by Leibniz rule for tangential derivatives and \eqref{Vnuzero}, that
    \begin{equation*}
        -\nabla V\,V\cdot\nu=-\nabla_T V\,V_T\,\cdot \nu=-\nabla_T (V\cdot \nu)\cdot V_T+\nabla_T\nu\,V_T\cdot V_T=\B(V_T,V_T)\,,
    \end{equation*}
    and the thesis follows by combining the two expressions above.
\end{proof}

In Section \ref{sec:dir} we will prove another version of Reilly's identity which would be suitable when one considers vector fields vanishing at the boundary. However, it seems to us that the type of weight that we are considering is not suitable for Dirichlet problems, see the discussion in Section \ref{sec:dir}.

\section{Proof of Theorem \ref{main:thmneu}}\label{sec:proof}
In this section, we provide the proof of Theorem \ref{main:thmneu}. We start with the following energy estimates for solutions to \eqref{eq:neu}.
\begin{lemma}\label{lemma:en1}
    Let $\Omega\subset \R^n$ be a bounded convex domain, $\vrho$ be a positive log-concave weight on $\Omega$, and let $u\in W^{1,p}(\Omega;\vrho)$ be a weak solution to
    \begin{equation}\label{fff}
        \begin{cases}
            -\mathrm{div}\big( \vrho\,|\nabla u|^{p-2}\nabla u\big)=\vrho\,f\quad & \text{in $\Omega$}
            \\
            \vrho\,\partial_\nu u=0\quad & \text{on $\partial \Omega$,}
        \end{cases}
    \end{equation}
such that  $\int_\Omega |u|^{p-2}u\,\vrho\,dx=0$, with $f\in L^{p'}(\Omega;\vrho)$ satisfying \eqref{f:compat}. Then we have
 \begin{equation}\label{energy:est}
     \int_\Omega |\nabla u|^p\,\vrho\,dx\leq C_0(p)\,d_\Omega^{p'}\int_\Omega |f|^{p'}\,\vrho\,dx\,.
 \end{equation}
\end{lemma}

\begin{proof}
    We test equation \eqref{fff} with $u$, and  use weighted Young's inequality and Poincar\'e inequality \eqref{poincare} to get
    \begin{align*}
            \int_\Omega |\nabla u|^p\,\vrho\,dx & =\int_\Omega f\,u\,\vrho\,dx
            \\
            &\leq \frac{1}{p'\delta^{p'}}\int_\Omega |f|^{p'}\,\vrho\,dx+\frac{\delta^{p}}{p}\int_\Omega u^p\,\vrho\,dx
            \\
            &\leq \frac{1}{p'\delta^{p'}}\int_\Omega |f|^{p'}\,\vrho\,dx+ \frac{C(p)\,d_\Omega^p\delta^{p}}{p}\int_\Omega |\nabla u|^p\,\vrho\,dx\,,
    \end{align*}
    for any $\delta>0$. Estimate \eqref{energy:est} follows by choosing $\delta=d_\Omega^{-1}\Big( \frac{p}{2\,C(p)}\Big)^{1/p}$.
\end{proof}

We now split the rest of the proof into three steps.

\noindent\emph{Step 1.} For the moment, we work under the assumption 
\begin{equation}\label{f:cinf}
    f\in C^\infty(\overline{\Omega})
\end{equation}
Let $u\in W^{1,p}(\Omega;\vrho)$ be a weak solution to problem \eqref{eq:neu}. Thanks to Lemma \ref{lemm:poincare}, we may assume that
\begin{equation}
    \int_\Omega |u|^{p-2}u\,\vrho\,dx=0\,.
\end{equation}
We consider a sequence of smooth convex domains $\Omega_m$, i.e., $\partial\Omega_m\in C^\infty$, such that $ \Omega_m\subset\subset \Omega_{m+1}\subset\subset \Omega$, and
\begin{equation}\label{conv:omm}
    \lim_{m\to\infty }\mathrm{dist}_\H(\Omega_m,\Omega)=0\,,\quad\lim_{m\to\infty} d_{\Omega_m}=d_\Omega\,.
\end{equation}
This may be done, for instance, by using \cite[Corollary 6.3.10]{kranz}. Set
\begin{equation}
    f_m=f-\Big(\int_{\Omega_m}\vrho\,dx\Big)^{-1}\int_{\Omega_m}\vrho f\,dx\,,
\end{equation}
In particular
\begin{equation}\label{fm:zero}
    \int_{\Omega_m}\vrho\,f_m\,dx=0\,,
\end{equation}
and 
$f_m\to f$ as $m\to\infty$ by \eqref{f:compat}.
Consider the function $u_m$ solution to
\begin{equation}\label{eq:neum}
\begin{cases}
    -\mathrm{div}\big( \vrho\,|\nabla u_m|^{p-2}\nabla u_m\big)=\vrho\,f_m\quad & \text{in $\Omega_m$}
    \\
    \partial_\nu u_m=0\quad &\text{on $\partial \Omega_m$}
    \end{cases}
\end{equation}

Recalling that $\vrho$ is strictly positive in $\Omega_m$, we have $W^{1,p}(\Omega_m;\vrho)=W^{1,p}(\Omega_m)$. Therefore, thanks to \eqref{fm:zero}, existence of $u_m\in W^{1,p}(\Omega_m)$ is classical, since it coincides with the minimizer of the functional
\begin{equation}
    \mathcal{F}_m(v)=\frac{1}{p}\int_{\Omega_m} \vrho\,|\nabla v|^p\,dx-\int_{\Omega_m} f_m\,v\,\vrho\,dx\,,\quad v\in W^{1,p}(\Omega_m)\,,
\end{equation}
and we may choose $u_m$ satisfying
\begin{equation}\label{zerommean}
    \int_{\Omega_m} |u_m|^{p-2}u_m\,\vrho\,dx=0\,.
\end{equation}
The following proposition encompasses a few regularity and convergence properties of $u_m$.
\begin{proposition}
    Let $u_m$ be the unique solution to \eqref{eq:neum} satisfying \eqref{zerommean}. Then
    \begin{equation}\label{c1m}
    u_m\in C^{1,\gamma_m}(\overline{\Omega}_m)
\end{equation}
for some $\gamma_m \in (0,1)$ and up to a subsequence, 
\begin{equation}\label{m:convc1}
    u_m\xrightarrow{m\to\infty} u\quad\text{in $C^1_{loc}(\Omega)$}
\end{equation}
\end{proposition}

\begin{proof}
Since $\vrho$ is strictly bounded away from zero in $\Omega_m$,  we may exploit \cite[Theorem 3.1 (a)]{cia90} and deduce that $u_m\in L^\infty(\Omega_m)$ up to an additive constant. Hence, the $C^{1,\gamma_m}$-regularity \eqref{c1m} follows by \cite[Theorem 2]{lieb88}.
    
Next, fix $\Omega'\Subset  \Omega''\Subset\Omega'''\Subset \Omega$; by \eqref{conv:omm} and \cite[Proposition 2.2.17]{henrot}, we may find $m_0>0$ such that
\begin{equation}\label{infogi}
   \Omega'\Subset  \Omega''\Subset\Omega'''\Subset \Omega_m\,,\quad\text{for all $m\geq m_0$.}
\end{equation} 
 Therefore, we may use the celebrated result of \cite{ser64} and a standard covering argument for $\Omega''$ to get
\begin{equation}\label{linfm}
    \|u_m\|_{L^\infty(\Omega'')}\leq C\, \left\{\left(\int_{\Omega'''}|u_m|^p\,dx\right)^{1/p}+1\right\}\,,
\end{equation}
for some constant $C>0$ independent on $m$. 
By using the the energy estimate \eqref{energy:est} with $u_m$, $\Omega_m$ in place of $u, \Omega$ respectively, and recalling that $d_{\Omega_m}\leq d_\Omega$, we deduce
\begin{equation}\label{en:estmm}
    \int_{\Omega_m}|\nabla u_m|^p\,\vrho\,dx\leq C\,\int_{\Omega}|f|^{p'}\,\vrho\,dx\,,
\end{equation}
with $C$ independent on $m$.
This piece of information coupled with  \eqref{infogi}, \eqref{conv:omm}, Poincar\'e inequality \eqref{poincare}
and the fact  $\inf_{\Omega'''}\vrho=\hat{c}>0$, entails
\begin{align}\label{yyy}
        \int_{\Omega'''}|u_m|^p\,dx& \leq \frac{1}{\hat{c}}\int_{\Omega_m}|u_m|^p\,\vrho\,dx\leq C\,\int_{\Omega_m}|\nabla u_m|^p\,\vrho\,dx\leq C'\,\int_{\Omega_m} |f|^{p'}\,\vrho\,dx\,,
\end{align}
with $C,C'>0$ independent on $m$.  Inequalities \eqref{linfm} and \eqref{yyy} entail the local estimate
\begin{equation}
    \|u_m\|_{L^\infty(\Omega'')}\leq C\,\quad\text{for all $m>m_0$,}
\end{equation}
for some constant $C>0$ independent on $m$. Hence, we may use the results of \cite{diben, tol} to deduce the existence of two constants $C_0>0$ and $\a\in (0,1)$ independent on $m$ such that
\begin{equation}
    \|u_m\|_{C^{1,\a}(\Omega')}\leq C_0\,,\quad\text{for all $m>m_0$, and all $\Omega'\subset \subset \Omega$.}
\end{equation}
By Ascoli-Arzel\'a's theorem, and a standard diagonal argument on $\Omega'$ and $u_m$, we may extract a subsequence $\{u_{m_k}\}_{k\in \N}$ such that
\begin{equation}\label{intanto}
    u_{m_k}\xrightarrow{k\to\infty} v\quad\text{in $C^1_{loc}(\Omega)$,}
\end{equation}
for some function $v\in C^1(\Omega)$. To conclude, it suffices to show that $v=u$ up to an additive constant.

First observe that by \eqref{en:estmm}, we have
\begin{equation}
    \int_{\Omega}|\A(\nabla u_m)|^{p'}\,\chi_{\Omega_m}\vrho\,dx=\int_{\Omega_m}|\nabla u_m|^{p}\,\vrho\,dx\leq C\,,
\end{equation}
namely, the sequence $\{\A(\nabla u_m)\,\chi_{\Omega_m}\}_{m\in \N}$ is bounded in $L^{p'}(\Omega;\vrho)$, where $\A$ is given by \eqref{def:A}. So we may extract a subsequence, still labeled as $u_m$, such that
\begin{equation}\label{conv:weaklp}
    \A(\nabla u_m)\,\chi_{\Omega_m}\to V\quad\text{weakly $L^{p'}(\Omega;\vrho)$\,,}
\end{equation}
for some vector field $V\in L^{p'}(\Omega;\vrho)$. On the other hand, by \eqref{intanto} and \eqref{conv:omm}, we have
\begin{equation}
     \A(\nabla u_{m_k})\chi_{\Omega_{m_k}}\xrightarrow{k\to\infty}\A(\nabla v)\quad\text{locally uniformly in $\Omega$\,, }
\end{equation}
hence $V=\A(\nabla v)$ in $\Omega$. Now take $\phi\in W^{1,p}(\Omega;\vrho)$ as a test function in \eqref{eq:neum}, that is
\begin{equation*}
    \int_{\Omega_{m_k}}\A(\nabla u_{m_k})\cdot\nabla\phi\,\vrho\,dx=\int_{\Omega_{m_k}}f_{m_k}\,\phi\,\vrho\,dx
\end{equation*}
so that, thanks to \eqref{conv:weaklp}, by letting $k\to\infty$ we obtain
\begin{equation}
   \int_\Omega \A(\nabla v)\cdot \nabla \phi\,\vrho\,dx= \int_\Omega |\nabla v|^{p-2}\nabla v\cdot \nabla \phi\,\vrho\,dx=\int_{\Omega} f\,\phi\,\vrho\,dx\,,
\end{equation}
for all $\phi\in W^{1,p}(\Omega;\vrho)$, that is $v$ is a weak solution to \eqref{eq:neu}, and thus $v=u$ up to an additive constant, hence the assertion is proved.
\end{proof}

We will need a further regularization procedure in order to apply the identities of  Section \ref{sec:fundlem}.
For each $m\in \N$, let $\{\rme\}_{\e>0}$ be the sequence of positive, log-concave functions constructed in Proposition \ref{prop:conv}.
Set
\begin{equation*}
    \fme=f-\Big(\int_{\Omega_m}\rme\,dx\Big)^{-1}\int_{\Omega_m}\rme f\,dx\,,
\end{equation*}
so that 
\begin{equation*}
    \int_{\Omega_m}\rme\,\fme\,dx=0\,,
\end{equation*}
and $\fme\xrightarrow{\e\to 0}f_m$ by definition of $f_m$ and \eqref{conv:omm}.
Then consider the sequence $\{\ume\}_{\e>0}=\{u_{\e,m}\}_{\e>0}$ of solutions to the boundary value problem
\begin{equation}\label{eq:neuue}
\begin{cases}
    -\mathrm{div}\Big( \rme\,\big[\e^2+|\nabla \ume|^2\big]^{\frac{p-2}{2}}\,\nabla \ume \Big)=\rme\,\fme\quad & \text{in $\Omega_m$}
    \\
    \partial_\nu \ume=0\quad & \text{on $\partial\Omega_m$,}
    \end{cases}
\end{equation}
satisfying
\begin{equation}
    \int_{\Omega_m}|\ume|^{p-2}\ume\,\vrho\,dx=0\,.
\end{equation}

Once again, existence of $\ume\in W^{1,p}(\Omega_m)$ follows from the fact that it coincides to the minimizer of the functional
\begin{equation}
    \mathcal{F}_{m,\e}(v)=\int_{\Omega_m}\frac{1}{p}\big[ \e^2+|\nabla v|^2\big]^\frac{p}{2}\,\rme\,dx-\int_{\Omega_m}\fme\,v\,\rme\,dx\,,\quad v\in W^{1,p}(\Omega_m)\,.
\end{equation}

Next we claim that
\begin{equation}\label{reg:ue}
    \ume\in C^\infty(\overline{\Omega}_m)\,,
\end{equation}
and that
\begin{equation}\label{ut:1}
    \nabla \ume\xrightarrow{\e\to 0} \nabla u_m\quad\text{uniformly in $\overline{\Omega}_m$.}
\end{equation}
First, for each $\e>0$, an application of \cite[Theorem 3.1 (a)]{cia90} ensures the existence of a proper additive constant $\cme$ such that
\begin{equation*}
    \|\ume-\cme\|_{L^\infty(\Omega_m)}\leq C\,,
\end{equation*}
with constant $C>0$ independent on $\e>0$.
It follows by \cite[Theorem 2]{lieb88}, that there exist two constants $C>0$  and $\theta\in (0,1)$ both independent on $\e>0$, such that 
\begin{equation}\label{C1ue}
    \|\nabla \ume\|_{C^{0,\theta}(\overline{\Omega}_m)}\leq C\,.
\end{equation}
Hence, we may extract a subsequence $\{\umek\}_k$ such that
\begin{equation}\label{ccc}
    \nabla \umek\xrightarrow{k\to\infty} \nabla w\quad\text{in $C^{0,\theta'}(\overline{\Omega}_m)$,}
\end{equation}
for all $0<\theta'<\theta$, and for some function $w\in C^{1,\theta}(\overline{\Omega}_m)$. We just need to show that $w= u_m$ up to an additive constant; to this end, we test equation \eqref{eq:neuue} with a function $\phi\in W^{1,p}(\Omega_m)$, as to obtain
\begin{equation*}
    \int_{\Omega_m} \A_{\e_k}(\nabla \umek)\cdot\nabla \phi\,\rmek\,dx=\int_{\Omega_m} \fmek\,\phi\,\rmek\,dx\,.
\end{equation*}
Thanks to \eqref{conv:Ae} and \eqref{ccc}, we have $\A_{\e_k}(\nabla \umek)\xrightarrow{k\to\infty} \A(\nabla w)$ uniformly in $\overline{\Omega}_m$. Hence, recalling \eqref{conv:vrhom} and that $\fme\xrightarrow{\e\to 0} f_m$ uniformly, by letting $k\to \infty $ in the above identity we find that $w$ satisfies
\begin{equation*}
    \int_{\Omega_m} \A(\nabla w)\cdot\nabla \phi\,\vrho\,dx=\int_{\Omega_m} f_m\,\phi\,\vrho\,dx\,,
\end{equation*}
that is $w$ is a weak solution to \eqref{eq:neum}, hence $w=u_m$ up to an additive constant. Since the above argument can be repeated to any subsequence $\{\umek\}$, from \eqref{ccc} we obtain convergence \eqref{ut:1} for the whole sequence $\{\ume\}_{\e>0}$.
\\

We now prove \eqref{reg:ue}. Via a standard difference quotient argument, adapted to (homogeneous) Neumann boundary condition, as in \cite[Theorem 8.2]{ben:fre}, \cite[pp. 270-277]{ladyz}, or \cite[Chapter 8.4]{giusti}, we deduce that
\begin{equation*}
    \ume\in W^{2,2}(\Omega_m)\,.
\end{equation*}
Thereby, by differentiating \eqref{eq:neuue} we obtain that $\ume$ is a strong solution to
\begin{equation*}
    -\mathrm{tr}\big( \rme\,\nabla_\xi \A_\e(\nabla \ume)\,D^2 \ume\big)-\nabla \rme\cdot \A_\e(\nabla \ume)=\rme\,\fme\quad\text{in $\Omega_m$\,}
\end{equation*}
where $\nabla_\xi$ denotes the differentiation with respect to the $\xi$ variable. Taking into account the smoothness of $\A_\e(\xi), \rme, \fme$ and \eqref{C1ue}, the regularity property \eqref{reg:ue} now follows via standard elliptic regularity theory for Neumann problems \cite{lieb:obl}.

\medskip

\noindent\emph{Step 2.} In this step we prove Theorem \ref{main:thmneu} under the assumption \eqref{f:cinf}.

We start by proving a bound for  $\nabla \A_\e(\nabla \ume)$ in the $L^2(\Omega_m;\rme)$-norm which is uniform with respect to $\e$ .

Thanks to the regularity properties $\partial \Omega_m\in C^\infty$, \eqref{vrho:m}, \eqref{reg:ue}, we are in the position to apply the identity \eqref{id:fundneu} with $V=\A_\e(\nabla \ume)$ and find
\begin{align}\label{aq:1}
     \int_{\Omega_m} \rme^{-1}\big[ \mathrm{div}\big(\rme\,\A_\e(\nabla \ume)\big)\big]^2\,dx= & \int_{\Omega_m}\rme\,\mathrm{tr}\big[(\nabla \A_\e(\nabla \ume))^2\big]\,dx
     \\
     &+\int_{\partial \Omega_m}\rme\,\B_m\big([\A_\e(\nabla \ume)]_T,[\A_\e(\nabla \ume)]_T\big)\,d\H^{n-1}\nonumber
     \\
     &+\int_{\Omega_m}\rme\,\nabla^2 h_\e\,\A_\e(\nabla \ume)\cdot \A_\e(\nabla \ume)\,dx\,,\nonumber
\end{align}
where $\B_m$ denotes the second fundamental form of $\partial \Omega_m$, and we set $h_\e=-\log \rme$, which is convex by the log-concavity of $\rme$ so that
\begin{equation}\label{aq:2}
    \int_{\Omega_m}\rme\,\nabla^2 h_\e\,\A_\e(\nabla \ume)\cdot \A_\e(\nabla \ume)\,dx\geq 0\,.
\end{equation}

By the convexity of $\Omega_m$, we also have
\begin{equation}\label{aq:3}
    \int_{\partial \Omega_m}\rme\,\B_m\big([\A_\e(\nabla \ume)]_T,[\A_\e(\nabla \ume)]_T\big)\,d\H^{n-1}\geq 0\,.
\end{equation}

Also, via the chain rule, \eqref{elem:ineq} and \eqref{eigen:2}, we estimate
\begin{equation}\label{aq:4}
    \mathrm{tr}\big[(\nabla \A_\e(\nabla \ume))^2\big]=\mathrm{tr}\big[(\nabla_\xi \A_\e(\nabla \ume)\,D^2 \ume)^2\big]\geq \overline{c}(p)\,|\nabla \A_\e(\nabla \ume)|^2\,,
\end{equation}
with explicit constant 
\begin{equation}\label{explicit}
    \overline{c}(p)=2\,\frac{\min\{p-1,1\}/\max\{p-1,1\}}{1+\big(\min\{p-1,1\}/\max\{p-1,1\}\big)^2}\,.
\end{equation}
In particular, we notice that $\overline{c}(2)=1$.

Coupling identities \eqref{aq:2}-\eqref{aq:4} with \eqref{aq:1}, and recalling equation \eqref{eq:neuue}, we get
\begin{equation}\label{fin1}
    \int_{\Omega_m}\rme\,f^2_\e\,dx\geq \overline{c}(p)\,\int_{\Omega_m}\rme\,|\nabla \A_\e(\nabla \ume)|^2\,dx\,.
\end{equation}
We let $\e\to 0$ in the above inequality so that, thanks to \eqref{conv:vrhom}, \eqref{ut:1} and \eqref{conv:Ae}, we obtain
\begin{equation}\label{fin2}
    \int_{\Omega_m}\vrho\,f^2_m\,dx\geq \overline{c}(p)\,\int_{\Omega_m}\vrho\,|\nabla \A(\nabla u_m)|^2\,dx\,.
\end{equation}
Now consider $\Omega'\Subset \Omega$; from \eqref{ut:1} and \eqref{fin2} we have
\begin{equation*}
    \overline{c}(p)\,\int_{\Omega'} \vrho\,|\nabla \A(\nabla u)|^2\,dx=\lim_{m\to\infty}   \overline{c}(p)\,\int_{\Omega'} \vrho\,|\nabla \A(\nabla u_m)|^2\,dx\leq \int_{\Omega}\vrho\,f^2\,dx\,,
\end{equation*}
and by letting $\Omega'\nearrow \Omega$ and using Fatou's lemma, it follows that
\begin{equation}\label{fin:4}
     \overline{c}(p)\,\int_{\Omega} \vrho\,|\nabla \A(\nabla u)|^2\,dx\leq \int_{\Omega}\vrho\,f^2\,dx\,. 
\end{equation}
Finally, by Poincar\'e inequality \eqref{poinc:compl}, H\"older's inequality and the estimates \eqref{energy:est}, \eqref{fin:4}, we deduce

\begin{align}\label{fin:5}
        \int_\Omega |\A(\nabla u)|^2\,\vrho\,dx & \leq C(n)\,d_\Omega^2\,\int_\Omega |\nabla \A(\nabla u)|^2\,\vrho\,dx+C(n)\,\left(\int_\Omega \vrho\,dx \right)^{-1}\left(\int_\Omega |\A(\nabla u)|\,\vrho\,dx \right)^2
        \\
        &\leq C(n)\,\overline{c}(p)^{-1}\,d_\Omega^2\,\int_\Omega \vrho\,f^2\,dx+C(n)\,\left(\int_\Omega\vrho\,dx\right)^{2/p-1}\,\left(\int_\Omega |\nabla u|^p\,\vrho\,dx \right)^{2/p'}\nonumber
        \\
        &\leq  C(n)\,\overline{c}(p)^{-1}\,d_\Omega^2\,\int_\Omega \vrho\,f^2\,dx+C(n,p)\,d_\Omega^2\left(\int_\Omega\vrho\,dx\right)^{2/p-1}\left( \int_\Omega|f|^{p'}\,\vrho\,dx\right)^{2/p'}\,.\nonumber
\end{align}
Inequalities \eqref{fin:4}-\eqref{fin:5} provide the thesis of Theorem \ref{main:thmneu} under the assumption \eqref{f:cinf}.
\\

\noindent\emph{Step 3.} Now we remove assumption \eqref{f:cinf}. By Lusin's Theorem, the set $C^0_c(\Omega)$ of continuous compactly supported functions in $\Omega$ is dense in $L^2(\Omega;\vrho)\cap L^{p'}(\Omega;\vrho)$. Hence, by a standard convolution argument, we may find a sequence $\{f_k\}_{k\in\N}\subset C^\infty(\overline{\Omega})$ such that
\begin{equation}\label{conv:fk}
    f_k\xrightarrow{k\to\infty} f\quad\text{in $L^2(\Omega;\vrho)\cap L^{p'}(\Omega;\vrho)$,}
\end{equation}
and satisfying $\int_\Omega \vrho\,f_k\,dx=0$ for all $k\in \N$. 
Let $u_k$ be the solution to the Neumann problem
\begin{equation}\label{eq:uk}
    \begin{cases}
        -\mathrm{div}(\vrho\,|\nabla u_k|^{p-2}\nabla u_k)=\vrho\,f_k\quad &\text{in $\Omega$}
        \\
        \vrho\,\partial_\nu u_k=0\quad & \text{on $\partial \Omega$,}
    \end{cases}
\end{equation}
satisfying $\int_\Omega |u_k|^{p-2}u_k\,\vrho\,dx=0$.
By the results in Steps 1 and 2, namely equations \eqref{fin:4}-\eqref{fin:5}, by Poincar\'e inequality \eqref{poincare}, the energy estimate \eqref{energy:est} on $u_k$ and \eqref{conv:fk}, we have
\begin{align}\label{fin:6}
        &\int_\Omega |\nabla \A(\nabla u_k)|^2\,\vrho\,dx  \leq C(p)\,\|f\|^2_{L^2(\Omega;\vrho)}
        \\
        &\int_\Omega | \A(\nabla u_k)|^2\,\vrho\,dx  \leq C(n,p,\vrho)\,d_\Omega^2\,\left\{\|f\|^2_{L^2(\Omega;\vrho)}+\|f\|^2_{L^{p'}(\Omega;\vrho)} \right\}\nonumber
        \\
        & \int_\Omega |u_k|^p\,\vrho\,dx\leq C(p)\,d_\Omega^p\int_\Omega |\nabla u_k|^{p}\,\vrho\,dx\leq C'(p)\,d_\Omega^{p+p'}\,\int_\Omega |f|^{p'}\,\vrho\,dx\,,\nonumber
\end{align}
for all $k\in \N$. Therefore, since $\vrho$ is strictly positive in $\Omega$, then $W^{1,q}_{loc}(\Omega;\vrho)=W^{1,q}_{loc}(\Omega)$ and we may extract a subsequence, still labeled as $u_k$, such that
\begin{align}
    & \A(\nabla u_k)\xrightarrow{k\to\infty}W\quad\text{weakly in $W^{1,2}(\Omega;\vrho)$ and in $L^{p'}(\Omega;\vrho)$,}
    \\
    & u_k\xrightarrow{k\to\infty} w \quad\text{weakly in $W^{1,p}(\Omega;\vrho)$, strongly in $L^p_{loc}(\Omega)$ and a.e. in $\Omega$.}\nonumber
\end{align}
In particular, $\A(\nabla u_k)\xrightarrow{k\to\infty} \A(\nabla w)$ almost everywhere on $\Omega$, hence standard real analysis results \cite[Theorem 13.44]{HS} tell us that $W=\A(\nabla w)$. Therefore, by taking a test function $\phi\in W^{1,p}(\Omega;\vrho)$ in equation \eqref{eq:uk} and letting $k\to\infty$ shows that $w$ is a weak solution to \eqref{eq:neu}, hence $w=u$ up to an additive constant. The proof is thus completed by letting $k\to\infty$ in estimates \eqref{fin:6} by exploiting the lower-semicontinuity of the $W^{1,2}(\Omega;\vrho)$-norm with respect to the weak convergence.

\section{Proof of Theorems \ref{thm:loc1}-\ref{thm:loc2}} \label{section_thm23}
The following section is devoted to the proof of Theorems \ref{thm:loc1}-\ref{thm:loc2} in  unbounded convex domain $\Omega$.  Since $\Omega$ is convex, we may find $\ov{R}>0$ and $K>0$ such that, up to a rotation, for any $0<R \leq \ov{R}$ and $x_0\in \partial \Omega$ we have 
\begin{align}\label{supergraph1}
&\Omega\cap \big(x_0+B'_{4R}\times (-K,K)\big)=\{x=(x',x_n):\, |x'|<4R,\,F(x')<x_n<K\}
\\
&\partial\Omega\cap \big(x_0+B'_{4R}\times (-K,K)\big)=\{x=(x',x_n):\, |x'|<4R,\,x_n=F(x')\}\,,\nonumber
\end{align}
for some  $L_\Omega$-Lipschitz convex function $F$. By taking $K$ and $R$ even smaller, we may assume that $K+2\,L_F\,R<1$ and $K>100\,n\,(1+L_\Omega)R$.
Let us consider the regularized functions
\begin{equation*}
    F_m(x')=F\ast \eta_m(x')+\|F-F\ast\eta_m\|_{L^\infty(B'_{3R})}+\frac{1}{m}\quad \ \textmd{ for }\,  |x'|<3R\,,
\end{equation*}
so that $F_m\in C^\infty(B'_{2R})$ is convex, $F_m>F$, with Lipschitz constant $L_{F_m}\leq L_\Omega$ and the sequence $\{F_m\}$ satisfies \eqref{f:fm}. For $R>0$ we  define the sets $\Omega_R,\Omega_{m,R}$ as  
\begin{align}\label{OmegaR}
\Omega_R & =\{x=(x',x_n):\, |x'|<R,\,F(x')<x_n<K\}
\\
    \Omega_{m,R} & =\{x=(x',x_n):\,|x'|<R,\,F_m(x')<x_n<K\}\,,\nonumber
\end{align}
which are convex and satisfy  $\Omega_{m,R}\subset \Omega_R$, and we also set
\begin{equation*}
    \mathrm{Graph}\, F_m=\big\{x=\big(x',F_m(x')\big),\,|x'|\leq R\big\}\,.
\end{equation*}
We now study energy estimates for solutions in the regularized domain, namely we consider $u_m\in W^{1,p}(\Omega_m;\vrho)=W^{1,p}(\Omega_m)$ solutions to
\begin{equation}\label{sol:umloc}
    \begin{dcases}
        -\mathrm{div}\big(\vrho\,|\nabla u_m|^{p-2}\nabla u_m\big)=\vrho\,f\quad &\text{in $\Omega_{m,R}$}
        \\
        \partial_\nu u_m=0\quad & \text{on $\mathrm{Graph}\,F_m$}
        \\
        u_m=u\quad & \text{on $\partial \Omega_{m,R}\setminus \mathrm{Graph}\,F_m $}\,.
    \end{dcases}
\end{equation}
Let us remark that, since $\vrho>0$ in $\Omega$, and $\partial \Omega_{m,R}\setminus \mathrm{Graph} F_m\subset \Omega$, we have that the trace of $u$ on such set is well defined. Therefore, existence and uniqueness of weak solutions $u_m$ to \eqref{sol:umloc} is classical. 
The next lemma provides energy estimates for $u_m$ which are uniform in $m\in \N$.

\begin{lemma}\label{lemma:en2}
    Let $u_m$ be the unique weak solution to \eqref{sol:umloc}. Then there exists a constant $C=C(n,p,\Omega,\vrho)$ such that, for every $m\in \N$, we have
    \begin{equation}\label{stimaloc:enm}
        \int_{\Omega_{m,R}}|\nabla u_m|^p\,\vrho\,dx\leq C\,\int_{\Omega_R}\big(|u|^p+|\nabla u|^p\big)\,\vrho\,dx+C\,\int_{\Omega_R}|f|^{p'}\,\vrho\,dx\,,
    \end{equation}
    and
    \begin{equation}\label{loc:enm1}
        \int_{\Omega_{m,R}} |u_m|^p\,\vrho\,dx\leq 2\,C\,\int_{\Omega_R}\big(|u|^p+|\nabla u|^p\big)\,\vrho\,dx+C\,\int_{\Omega_R}|f|^{p'}\,\vrho\,dx.
    \end{equation}
\end{lemma}

\begin{proof}
    We test \eqref{sol:umloc} with $u_m-u$, and by using Young's inequality we get
\begin{align*}
        \int_{\Omega_{m,R}}|\nabla u_m|^p\,\vrho\,dx \leq &  \int_{\Omega_{m,R}} \vrho\,|\nabla u_m|^{p-1}\,|\nabla u|\,dx+\int_{\Omega_{m,R}}|f|\,|u_m-u|\,\vrho\,dx
        \\
        \leq & \frac{1}{4}\int_{\Omega_{m,R}} |\nabla u_m|^{p}\,\vrho\,dx+C(p)\,\int_{\Omega_{m,R}}|\nabla u|^p\,\vrho\,dx
        \\
        &+\frac{\delta^p}{p}\int_{\Omega_{m,R}}|u_m-u|^p\,\vrho\,dx+\frac{1}{\delta^{p'}p'}\int_{\Omega_{m,R}} |f|^{p'}\,\vrho\,dx\,,
\end{align*}
    for any $\delta>0$. Furthermore, $u_m-u$ satisfies the hypothesis of Theorem \ref{thm:zerotrpo}, so that
    \begin{equation*}
        \int_{\Omega_{m,R}}|u-u_m|^p\,\vrho\,dx\leq C_0\,\int_{\Omega_{m,R}} |\nabla (u-u_m)|^p\,\vrho\,dx\,,
    \end{equation*}
    for a constant $C_0$ depending on $n,p,\Omega,\vrho$. Connecting the two inequalities displayed above with the choice of $\delta^p=p\,C_0/4$, and recalling  that  $\Omega_{m,R}\subset \Omega_R$, we obtain  \eqref{stimaloc:enm}. Finally, to get \eqref{loc:enm1} it suffices to couple Poincar\'e inequality \eqref{poinc:zero}, applied to $u_m-u$, and \eqref{stimaloc:enm}.
\end{proof}

The next lemma provides regularity for $u$ and $u_m$ under the additional assumption of boundedness of $f$.

\begin{lemma}
    Let $p\in (1,\infty)$, and let $u$ be a local weak solution to \eqref{eq:neuloc} and $u_m$ be the weak solution to \eqref{sol:umloc}. Assume that $f\in L^\infty(\Omega_{R})$. Then
    \begin{equation}
        u,u_m\in C^{1}\big(\overline{\Omega}_{m,R/2}\cap \{|x_n|\leq K/2\}\big)\,,
    \end{equation}
 and we have that
    \begin{equation}\label{conve2}
        u_m\xrightarrow{m\to\infty} u\quad\text{in $C^{1}_{loc}(\Omega_{R})$}.
    \end{equation}
\end{lemma}

\begin{proof}
    Observe that $\Omega_{m,R}\Subset \Omega$, and recalling that $\vrho>0$ is strictly bounded away from zero in $\Omega$, it follows by \cite{ser64} that $u\in L^\infty_{loc}(\Omega)$. Hence, thanks to \cite{diben, tol}, we have that  $u\in C^{1,\beta_m}(\overline{\Omega}_{m,R})$,
    for some $\beta_m\in (0,1)$.
Furthemore, \cite{ser64} tells us that in every subset $\Omega'\Subset \Omega''\Subset\Omega'''\Subset\Omega_{m,R}$, we have that $u_m$ is bounded with estimate
\begin{equation}\label{ser:loc}
    \|u_m\|_{L^\infty(\Omega'')}\leq C\left(\int_{\Omega'''}|u_m|^p\,dx \right)^{1/p}+C\leq \frac{C}{\inf_{\Omega'''}\vrho}\,\left(\int_{\Omega_{m,R}}|u_m|^p\,\vrho\,dx \right)^{1/p}+C\leq C'
\end{equation}
with $C,C'$ independent on $m$, where in the second inequality we used \eqref{loc:enm1}. It then follows by \cite{diben, tol} that $u\in C^{1,\theta}(\Omega')$, with 
\begin{equation}
    \|u_m\|_{C^{1,\theta}(\Omega')}\leq C
\end{equation}
where $C>1$ and $\theta\in (0,1)$ are independent on $m$. Henceforth, since $\Omega_{m,R}\nearrow \Omega_R$ as $m\to\infty$, it follows by Ascoli-Arzel\'a Theorem that we may extract a subsequence, still labeled as $u_m$, such that $u_m\xrightarrow{m\to\infty} v$ in $C^1_{loc}(\Omega_{R})$, with $v\in W^{1,p}(\Omega_R;\vrho)$ as a consequence of \eqref{stimaloc:enm}. Testing \eqref{sol:umloc} with any test function $\vphi\in C^\infty_c(\R^n)$, $\vphi=0$ on $\partial \Omega_R\setminus\mathrm{Graph}\,F $, which is admissible since $\partial \Omega_{m,R}\setminus \mathrm{Graph}\,F_m\subset \partial \Omega_R\setminus\mathrm{Graph}\,F $ by construction, and letting $m\to\infty$ we deduce that $v$ is solution to
\begin{equation*}
    \begin{cases}
             -\mathrm{div}\big(\vrho\,|\nabla v|^{p-2}\nabla v\big)=\vrho\,f\quad &\text{in $\Omega_{R}$}
        \\
        \partial_\nu v=0\quad & \text{on $\mathrm{Graph}\,F$}
        \\
        v=u\quad & \text{on $\partial \Omega_{R}\setminus \mathrm{Graph}\,F $}\,.
    \end{cases}
\end{equation*}
whence $v=u$ by uniqueness.  By repeating the same argument for any subsequence of $\{u_m\}_{m\in \N}$, we infer that any subsequence $u_{m_j}$ converges to $u$, thus \eqref{conve2} is proven. 

We are left to show the regularity of $u_m$ up to the boundary. To this end, we first need to prove that $u_m$ is bounded up to the graph of $F_m$. Since we could not find a precise reference, we quickly sketch the proof, which proceeds via Moser iteration.

Let  $M>1$ be fixed, let $x_{0,m}\in \mathrm{Graph}\,F_m$ and set $$\overline{u}_m=(u_m)_++1,\quad v_m=\min\{\overline{u}_m,M\}.$$

For $0<r_1<r_2<R/10^n$, let $\phi$ be a cut-off function satisfying $\phi\in C^\infty_c(B_{r_2}(x_{0,m}))$, $0\leq \phi\leq 1$, $\phi\equiv 1$ in $B_{r_1}(x_{0,m})$, $|\nabla \phi|\leq 8/(r_2-r_1)$.
Let $t>0$ and take $\vphi=\big\{v_m^{t}\overline{u}_m-1\big\}\,\phi^p$ as a test function in \eqref{sol:umloc}, thus obtaining
\begin{align*}
t\,\int_{\Omega_{m,R}}\vrho&\,|\nabla v_m|^p\,v_m^{t-1}\overline{u}_m\,\phi^p\,dx+\int_{\Omega_{m,R}}\vrho\,v_m^t\,|\nabla \overline{u}_m|^p\,\phi^p\,dx
\\
\leq &\,p\,\int_{\Omega_{m,R}}\vrho\,|\nabla \overline{u}_m|^{p-1}\,v_m^t\,\overline{u}_m\,\phi^{p-1}\,|\nabla \phi|\,dx +\|f\|_\infty\,\int_{\Omega_{m,R}}\vrho\,v_m^t\,\overline{u}_m\,\phi^p\,dx
    \\
    \leq &\,\frac{1}{2}\int_{\Omega_{m,R}}\vrho\,|\nabla \overline{u}_m|^p\,v_m^t\,\phi^p\,dx+C(p)\,\int_{\Omega_{m,R}}\vrho\,v_m^t\,\overline{u}_m^p\bigg\{|\nabla \phi|^p+\|f\|_\infty\,\phi^p\, \bigg\}\,dx\,,
\end{align*}
where in the last inequality we used Young's inequality and the fact that $\overline{u}_m\geq 1$.

Recalling that $\inf_{\Omega_{m,R}}\vrho>0$ and $\vrho$ is bounded, and noticing that $|\nabla v_m|^p\,v_m^{t-p}\overline{u}_m^p=|\nabla v_m| v_m^{t-1}\overline{u}_m$ since $v_m=\overline{u}_m$ in $\{\nabla v_m\neq 0\}$, the above inequality entails
\begin{align*}
     \int_{\Omega_{m,R}}|\nabla (v_m^{t/p}\,\overline{u}_m)|^p\,\phi^p\,dx\leq C\,t^{p-1}\,\int_{\Omega_{m,R}}\big|v_m^{t/p}\,\overline{u}_m\big|^p\bigg\{\phi^p+|\nabla \phi|^p\bigg\}\,dx\,.
\end{align*}
Thus, by the properties of the cut-off function $\phi$, and exploiting  Sobolev inequality on balls centered at points of $\mathrm{Graph}\,F_m$, whose quantitative constant depends on the Lipschitz constant $L_F$, we deduce
\begin{equation}\label{moser1}
    \left(\int_{\Omega_{m,R}\cap B_{r_1}(x_{0,m})}v_m^{q\,(t+p)}\,dx\right)^{1/q}\leq \frac{C}{(r_2-r_1)^p}\,t^{p-1}\,\int_{\Omega_{m,R}\cap B_{r_2}(x_{0,m})} \overline{u}_m^{t+p}\,dx\,,
\end{equation}
provided the right hand side is finite, where we also used that $v_m\leq \overline{u}_m$, and we set
\begin{equation*}
    q=\begin{cases}
        \frac{n}{n-p}\quad &1<p<n
        \\
        \text{any number $>1$}\quad & p\geq n\,.
    \end{cases}
\end{equation*}
We let $M\to \infty$ so that $v_m \nearrow \overline{u}_m$ in \eqref{moser1}, and set $\gamma=t+p$ as to obtain
\begin{equation}\label{moser2}
    \| \overline{u}_m\|_{L^{q\,\gamma}(\Omega_{m,R}\cap B_{r_1}(x_{0,m}))}\leq C\,\left(\frac{(\gamma-p)^{p-1}}{(r_2-r_1)^{p}}\right)^{1/\gamma}\,\| \overline{u}_m\|_{L^{\gamma}(\Omega_{m,R}\cap B_{r_2}(x_{0,m}))}\,,
\end{equation}
provided the right hand side is finite. We are in the position to apply Moser iteration; for $k=0,1,2,\dots$, set $\gamma_k=q^k\,p$,  and $r_k=R_0(1+2^{-k})$ for any fixed radius $R_0\leq R/(100\sqrt{n}L_F)$. Inequality \eqref{moser2} thus becomes
\begin{align*}
\|\overline{u}_m\|_{L^{\gamma_{k+1}}(\Omega_{m,R}\cap B_{r_{k+1}}(x_{0,m}))}&\leq C\,\left(\frac{2p}{R_0} \right)^{1/q^k}\,q^{k/q^k}\,\|\overline{u}_m\|_{L^{\gamma_{k}}(\Omega_{m,R}\cap B_{r_{k}}(x_{0,m}))}
    \\
   &\leq C^{1/q^k}\,q^{k/q^k}\|\overline{u}_m\|_{L^{\gamma_{k}}(\Omega_{m,R}\cap B_{r_{k}}(x_{0,m}))}\,.
\end{align*}
Iterating yields
\begin{align*}\|\overline{u}_m\|_{L^{\gamma_{k+1}}(\Omega_{m,R}\cap B_{r_{k+1}}(x_{0,m}))}\leq C^{\sum_k 1/q^k}\,q^{\sum_k k/q^k}\,\|\overline{u}_m\|_{L^p(\Omega_{m,R}\cap B_{2R_0}(x_{0,m}))}
    \\
    C'\,\|\overline{u}_m\|_{L^p(\Omega_{m,R}\cap B_{2R_0}(x_{0,m}))}\,,
\end{align*}
hence letting $k\to\infty$, and recalling the definition of $\overline{u}_m$, we deduce that for any radius $R_0$ fixed, we have
\begin{equation}\label{moser:fin}
    \|(u_m)_+\|_{L^\infty(\Omega_{m,R}\cap B_{R_0}(x_{0,m}))}\leq C\,\left(\int_{\Omega_{m,R}\cap B_{2R_0}(x_{0,m})}(u_m)_+^p\,dx+1\right)^{1/p}\,.
\end{equation}
Repeating the same argument for $(u_m)_-$ shows that $u_m$ is bounded in a neighbourhood of $\mathrm{Graph}\,F_m$. Coupling this information with \eqref{loc:enm1} and \eqref{ser:loc}, a standard covering argument tells that 
\begin{equation}\label{um:bounded}
    \|u_m\|_{L^\infty(\overline{\Omega}_{m,2/3R}\cap \{|x_n|\leq 2/3K)}\leq C_m\,,
\end{equation}
where $C_m$ depends also on $\inf_{\Omega_{m,R}}\vrho$\,. Then, by \cite[Theorem 2]{lieb88}, we finally deduce that $u_m\in C^{1,\beta_m}\big(\overline{\Omega}_{m,R/2}\cap \{|x_n|\leq K/2\}\big)$ coupled with the quantitative estimate
\begin{equation}\label{C1um:est}
    \|u_m\|_{C^{1,\beta_m}\big(\overline{\Omega}_{m,R/2}\cap \{|x_n|\leq K/2\}\big)}\leq C'_m\,,
\end{equation}
for some $\beta_m\in (0,1)$ and $C_m>0$, thus completing the proof.
\end{proof}

\medskip

\begin{proof}[Proof of Theorem \ref{thm:loc1}] $ $

\emph{Step 1.} In this step we assume that \eqref{f:cinf} is in force.
Let $u$ and $u_m$ be weak solutions to \eqref{eq:neuloc} and \eqref{sol:umloc} respectively. Fix $m\in \N$, and for all $\e>0$, let $\rme$ be the regularization of $\vrho$ given by Proposition \ref{prop:conv}, and consider $\{u_{\e,m}\}_{\e>0}\subset W^{1,p}(\Omega_{m,R})$ be the family of solutions to
\begin{equation}\label{equa:locue}
    \begin{cases}
        -\mathrm{div}\big(\rme\big[\e^2+|\nabla \ume|^2 \big]^{\frac{p-2}{2}}\nabla \ume \big)=\rme\,f\quad &\text{in $\Omega_{m,R}$}
        \\
        \partial_\nu \ume=0\quad &\text{on $\mathrm{Graph} F_m$}
        \\
        \ume=u\quad&\text{on $\partial \Omega_{m,R}\setminus \mathrm{Graph} F_m $.}
    \end{cases}
\end{equation}
By arguing as in the the proof of \eqref{stimaloc:enm}-\eqref{loc:enm1}, and exploiting the fact that $\rme\xrightarrow{\e\to 0} \vrho$ uniformly in $\Omega_{m,R}$, one infers
\begin{align}\label{loc:estue}
    \int_{\Omega_{m,R}} |\nabla \ume|^p\,\vrho\,dx & \leq C_m\,\int_{\Omega_R}\big( |u|^p+|\nabla u|^p\big)\,\vrho\,dx+C_m\,\int_{\Omega_R}|f|^{p'}\vrho\,dx
    \\
    \int_{\Omega_{m,R}} |\ume|^p\,\vrho\,dx & \leq C_m\,\int_{\Omega_R}\big( |u|^p+|\nabla u|^p\big)\,\vrho\,dx+C_m\,\int_{\Omega_R}|f|^{p'}\vrho\,dx\,,\nonumber
\end{align}
with constant $C_m$ independent on $\e$.
Also, by repeating verbatim the proof of \eqref{C1um:est} with $\ume$ in place of $u_m$, one deduces that $\ume\in C^{1,\beta_m}\big(\overline{\Omega}_{m,R/2})\cap \{|x_n|\leq K/2 \}\big)$, with quantitative estimate
\begin{equation}\label{C1:locue}
    \|\ume\|_{C^{1,\beta_m}\big(\overline{\Omega}_{m,R/2}\cap \{|x_n|\leq K/2 \}\big)}\leq C_m\,
\end{equation}
with $\beta_m\in (0,1)$ and $C_m$  independent on $\e>0$. From estimates \eqref{loc:estue}-\eqref{C1:locue} we deduce that, up to a subsequence, $\ume\xrightarrow{\e\to 0} v_m$ weakly in $W^{1,p}(\Omega_{m,R})$ and strongly in $C^{1,\beta_m'}\big(\overline{\Omega}_{m,R/2}\cap \{|x_n|\leq K/2 \}\big)$ for all $0<\beta'_m<\beta_m$, for some function $v_m\in W^{1,p}(\Omega_{m,R})$. Testing equation \eqref{equa:locue} with any function $\vphi\in C^\infty_c(\R^n)$, $\vphi=0$ on $\partial \Omega_{m,R}\setminus \mathrm{Graph} F_m$, and letting $\e\to 0$ shows that $v_m$ is solution to \eqref{sol:umloc}, hence $v_m=u_m$. We have thus shown that 
\begin{equation}\label{cisiamm}
    \ume\xrightarrow{\e\to 0} u_m\quad\text{weakly in $W^{1,p}(\Omega_{m,R})$, strongly in $C^1\big(\overline{\Omega}_{m,R/2}\cap \{|x_n|\leq K/2 \}\big)$.}
\end{equation}
Next, a difference quotient argument adapted to homogeneous Neumann problems as in \cite[Theorem 8.2]{ben:fre}, \cite[pp. 270-277]{ladyz} or \cite[Chapter 8.4]{giusti}, one finds that $\ume\in W^{2,2}\big(\Omega_{m,R/2}\cap \{|x_n|\leq K/2 \}\big)$. Hence, by differentiating \eqref{equa:locue} we find that $\ume$ is a strong solution to
\begin{equation*}
    -\mathrm{tr}\big( \rme\,\nabla_\xi \A_\e(\nabla \ume)\,D^2 \ume\big)-\nabla \rme\cdot \A_\e(\nabla \ume)=\rme\,\fme\quad\text{a.e. in $\Omega_{m,R/2}\cap \{|x_n|\leq K/2 \}$,}
\end{equation*}
 with Neumann boundary condition $\partial_\nu \ume=0$ on $\mathrm{Graph} F_m$. Since $\ume\in C^{1,\beta_m}\big(\overline{\Omega}_{m,R/2}\cap \{|x_n|\leq K/2 \}\big)$, we have 
$\A_\e(\nabla \ume)\in C^{0,\beta_m}\big(\overline{\Omega}_{m,R/2}\cap \{|x_n|\leq K/2 \}\big)$, hence 
standard elliptic regularity for conormal problems \cite{lieb:obl} ensures that
\begin{equation}\label{loc:regCinf}
    \ume\in C^\infty\big(\overline{\Omega}_{m,R/4}\cap \{|x_n|\leq K/4\}\big)\,.
\end{equation}

\noindent\emph{Step 2.} Fix $R_0\leq R/100$, $x_0\in \mathrm{Graph}\,F$, and a cut-off function $\vphi \in C^\infty_c(B_{R_0}(x_0)) $ such that
\begin{equation*}
    \vphi\equiv 1 \quad\text{in $B_{R_0/2}(x_0)$,}\quad |\nabla \vphi|\leq C(n)/R_0\,.
\end{equation*}
We apply \eqref{id:fundneu00} with the choice of $V=\A_\e(\nabla \ume)$
 and $\phi=\vphi^2$, by using \eqref{equa:locue} we obtain
\begin{align}\label{alp}
\int_{\Omega_{m,R}}\rme\,f^2\,\vphi^2\,dx  = &\,\int_{\Omega_{m,R}} \rme\,\mathrm{tr}\big[\big(\nabla \A_\e(\nabla \ume)\big)^2 \big]\,\vphi^2\,dx
    \\
    &+\int_{\mathrm{Graph}\,F_m}\rme\,\B_m\big([\A_\e(\nabla \ume)]_T,[\A_\e(\nabla \ume)]_T\big)\,\vphi^2\,d\H^{n-1}\nonumber
     \\
     &+\int_{\Omega_{m,R}}\rme\,\nabla^2 h_\e\,\A_\e(\nabla \ume)\cdot \A_\e(\nabla \ume)\,\vphi^2\,dx\nonumber
     \\
     &-2\,\int_{\Omega_{m,R}} (\A_\e(\nabla \ume)\cdot \nabla \vphi)\,\mathrm{div}\big(\rme\,\A_\e(\nabla \ume)\big)\,\vphi\,dx\nonumber
     \\
     &+2\,\int_{\Omega_{m,R}}\rme\,\nabla \A_\e(\nabla \ume)\,\A_\e(\nabla \ume)\cdot\nabla \vphi\,\vphi\,dx.\nonumber
\end{align}
By \eqref{aq:4} and the convexity of $F_m$ and $h_\e$, we have
\begin{align}
    &\int_{\Omega_{m,R}} \rme\,\mathrm{tr}\big[\big(\nabla \A_\e(\nabla \ume)\big)^2 \big]\,\vphi^2\,dx\geq \ov{c}(p)\,\int_{\Omega_{m,R}}\rme\,\big|\nabla \A_\e(\nabla \ume)\big|^2 \,\vphi^2\,dx
    \\
    &\int_{\mathrm{Graph}\,F_m}\rme\,\B_m\big([\A_\e(\nabla \ume)]_T,[\A_\e(\nabla \ume)]_T\big)\,\vphi^2\,d\H^{n-1}\geq 0\nonumber
    \\
    &\int_{\Omega_{m,R}}\rme\,\nabla^2 h_\e\,\A_\e(\nabla \ume)\cdot \A_\e(\nabla \ume)\,\vphi^2\,dx\geq 0\,.\nonumber
\end{align}
Moreover, from Young's inequality and \eqref{equa:locue}, we find
\begin{align}
2\,\left|\int_{\Omega_{m,R}} (\A_\e(\nabla \ume)\cdot \nabla \vphi)\,\mathrm{div}\big(\rme\,\A_\e(\nabla \ume)\big)\,\vphi\,dx\right|\leq & \int_{\Omega_{m,R}} \rme^{-1}\mathrm{div}\big(\rme\,\A_\e(\nabla \ume)\big)^2\,\vphi^2\,dx
        \\
        &+\int_{\Omega_{m,R}} \rme\,|\A_\e(\nabla \ume)|^2\,|\nabla \vphi|^2\,dx\nonumber
        \\
        = \int_{\Omega_{m,R}} \rme\,f^2\,\vphi^2\,dx  +C\,\int_{\Omega_{m,R}} \rme & \,|\A_\e(\nabla \ume)|^2\,|\nabla \vphi|^2\,dx\,,\nonumber
\end{align}
and
\begin{align}\label{alp1}
2\,\left|\int_{\Omega_{m,R}}\rme\,\nabla \A_\e(\nabla \ume)\,\A_\e(\nabla \ume)\cdot\nabla \vphi\,\vphi\,dx\right|\leq &  \frac{\overline{c}(p)}{2}\,\int_{\Omega_{m,R}}\rme\,|\nabla \A_\e(\nabla \ume)|^2\,\vphi^2\,dx
        \\
        &+C\,\int_{\Omega_{m,R}}\rme\,|\A_\e(\nabla \ume)|^2\,|\nabla \vphi|^2\,dx \,.\nonumber
\end{align}
Combining \eqref{alp}-\eqref{alp1}, we obtain
\begin{equation}\label{alp2}
    \int_{\Omega_{m,R}}\rme\,|\nabla \A_\e(\nabla \ume)|^2\,\vphi^2\,dx\leq C\,\int_{\Omega_{m,R}}\rme\,f^2\,\vphi^2\,dx+C\,\int_{\Omega_{m,R}}\rme\,|\A_\e(\nabla \ume)|^2\,|\nabla \vphi|^2\,dx\,.
\end{equation}
From \eqref{conv:Ae}, \eqref{conv:vrhom} and \eqref{cisiamm}, we have that $\A_\e(\nabla \ume)\xrightarrow{\e\to 0} \A(\nabla u_m)$ uniformly in $\Omega_{m,R}\cap B_{2R_0}(x_0)$, and the right hand side of \eqref{alp2} is uniformly bounded in $\e>0$. It follows that, up to a subsequence, we have $\A_\e(\nabla \ume)\xrightarrow{\e\to 0} \A(\nabla u_m)$ weakly in $W^{1,2}(\Omega_{m,R}\cap B_{R_0}(x_0))$, so by letting $\e\to 0$ in \eqref{alp2} and using the lower-semicontinuity of the norm and the properties of $\vphi$, we infer
\begin{equation}\label{alp3}
     \int_{\Omega_{m,R}\cap B_{R_0/2}(x_0)}\vrho\,|\nabla \A(\nabla u_m)|^2\,dx\leq C\,\int_{\Omega_{m,R}\cap B_{R_0}(x_0)}\vrho\,f^2\,dx+\frac{C}{R_0^2}\,\int_{\Omega_{m,R}\cap B_{R_0}(x_0)}\vrho\,|\A(\nabla u_m)|^2\,dx\,.
\end{equation}
On the other hand, since $2(p-1)\leq p$, by H\"older's inequality and \eqref{stimaloc:enm} we get
\begin{align}\label{alp4}
        \int_{\Omega_{m,R}\cap B_{R_0}(x_0)}\vrho\,|\A(\nabla u_m)|^2\,dx & \leq C_0\,\left( \int_{\Omega_{m,R}\cap B_{R_0}(x_0)}\vrho\,|\nabla u_m|^{p}\,dx\right)^{2(p-1)/p}
        \\
        &\leq C_1\,\left(\int_{\Omega_R}\vrho\,\left\{|u|^p+|\nabla u|^p\right\}\,dx\right)^{2(p-1)/p}+C_1\,\left(\int_{\Omega_R}\vrho\,|f|^{p'}\,dx\right)^{2(p-1)/p}\,,\nonumber
\end{align}
with constants $C_0,C_1$ depending on $n,p,\vrho,R_0$. Thanks to inequalities \eqref{alp3}-\eqref{alp4}, and the argument of Remark \ref{remark:utile}, we have that the sequence $\{\A(\nabla u_m)\}_{m\in \N}$ is bounded in $W^{1,2}(\Omega_R\cap B_{R_0/2}(x_0);\vrho)$, uniformly in $m\in \N$. 

This information coupled with \eqref{conve2} and the fact that 
$\Omega_{m,R}\nearrow \Omega_R$ as $m\to \infty$ enable us to deduce that, up to subsequence, we have
\begin{equation*}
    \A(\nabla u_m)\xrightarrow{m\to \infty} \A(\nabla u)\quad \text{weakly in $W^{1,2}(\Omega_R\cap B_{R_0}(x_0))$.}
\end{equation*}
Therefore, by letting $m\to \infty$ in \eqref{alp3} we obtain 
\begin{align}\label{quasifin}
        \int_{\Omega_{R}\cap B_{R_0/2}(x_0)}\vrho\,|\nabla \A(\nabla u)|^2\,dx  \leq C\,\int_{\Omega_{R}\cap B_{R_0}(x_0)}\vrho\,f^2\,dx+\frac{C}{R_0^2}\,\int_{\Omega_{R}\cap B_{R_0}(x_0)}\vrho\,|\A(\nabla u)|^2\,dx\,,
\end{align}
with $C$ depending on $n,p$, whereas by H\"older's inequality, we deduce
\begin{equation}\label{alppp4}
    \int_{\Omega_{R}\cap B_{R_0}(x_0)}\vrho\,|\A(\nabla u)|^2\,dx 
        \leq \left(\int_{\Omega_{R}\cap B_{R_0}(x_0)}\vrho\,dx\right)^{(2-p)/p}\,\left(\int_{\Omega_{R}\cap B_{R_0}(x_0)}|\nabla u|^p\,\vrho\,dx\right)^{2/p'}\,.
\end{equation}

\noindent\emph{Step 3} Now we remove assumption \eqref{f:cinf}. Let $\{f_k\}_{k\in \N}$ be a sequence of $C^\infty_c(\R^n)$ functions such that
\begin{equation}\label{lo:fktof}
    f_k\xrightarrow{k\to\infty} f\quad\text{in $L^{p'}(\Omega;\vrho)$}\,.
\end{equation}
Let $u_k$ be the unique solution to
\begin{equation}\label{eq:locuk}
    \begin{cases}
          -\mathrm{div}\big(\vrho\,|\nabla u_k|^{p-2}\nabla u_k\big)=\vrho\,f_k\quad &\text{in $\Omega_{R}$}
        \\
        \partial_\nu u_k=0\quad & \text{on $\mathrm{Graph}\,F$}
        \\
        u_k=u\quad & \text{on $\partial \Omega_{R}\setminus \mathrm{Graph}\,F $}\,.
    \end{cases}
\end{equation}
By testing \eqref{eq:locuk} with $u-u_k$, repeating analogous computations as \eqref{stimaloc:enm}-\eqref{loc:enm1}, and using \eqref{lo:fktof} we find the energy estimate
\begin{align}\label{loc:enestuk}
&\int_{\Omega_{R}}|\nabla u_k|^p\,\vrho\,dx\leq C\,\int_{\Omega_R}\big(|u|^p+|\nabla u|^p\big)\,\vrho\,dx+C\,\int_{\Omega_R}|f|^{p'}\,\vrho\,dx\,,
\\
       & \int_{\Omega_{R}} |u_k|^p\,\vrho\,dx\leq C\,\int_{\Omega_R}\big(|u|^p+|\nabla u|^p\big)\,\vrho\,dx+C\,\int_{\Omega_R}|f|^{p'}\,\vrho\,dx,\nonumber
\end{align}
which is independent on $k$. Therefore, from Theorem \ref{thm:cpt} there exists a subsequence, still labeled as $u_k$, such that
\begin{equation*}
    u_{k}\xrightarrow{k\to\infty} v\quad\text{strongly in $L^p(\Omega_R;\vrho)$, weakly in $W^{1,p}(\Omega_R;\vrho)$ and a.e. in $\Omega_R$.}
\end{equation*}
Moreover, from the results of the previous step, namely \eqref{quasifin} and \eqref{alppp4}, and by using \eqref{lo:fktof} we deduce
\begin{align}\label{quasifin2}
        &\int_{\Omega_{R}\cap B_{R_0/2}(x_0)}\vrho\,|\nabla \A(\nabla u_k)|^2\,dx  \leq C\,\int_{\Omega_{R}\cap B_{R_0}(x_0)}\vrho\,f^2\,dx+\frac{C}{R_0^2}\,\int_{\Omega_{R}\cap B_{R_0}(x_0)}\vrho\,|\A(\nabla u_k)|^2\,dx
        \\
        &\int_{\Omega_{R}\cap B_{R_0}(x_0)}\vrho\,|\A(\nabla u_k)|^2\,dx 
        \leq \left(\int_{\Omega_{R}\cap B_{R_0}(x_0)}\vrho\,dx\right)^{(2-p)/p}\,\left(\int_{\Omega_{R}\cap B_{R_0}(x_0)}|\nabla u_k|^p\,\vrho\,dx\right)^{2/p'}\,.\nonumber
\end{align}
with $C$ depending on $n,p$. Hence, by \eqref{loc:enestuk} and \eqref{quasifin2}, we may extract a subsequence, still labeled as $u_k$, such that $\A(\nabla u_k)\xrightarrow{k\to\infty} \A(\nabla u)$ weakly in $W^{1,2}(\Omega_R\cap B_{R_0/2(x_0)};\vrho)$ and $L^2(\Omega_R\cap B_{R_0(x_0)};\vrho)$, and $u_k\xrightarrow{k\to\infty}u$ in $W^{1,p}(\Omega_R\cap B_{R_0}(x_0);\vrho)$. Whence, by letting $k\to \infty$ in \eqref{quasifin2} we get the thesis. 
\end{proof}

\begin{proof}[Proof of Theorem \ref{thm:loc2}] The proof of Theorem \ref{thm:loc2} in the case $p>2$ reproduces for the most part the proof of Theorem \ref{thm:loc1}. The main difference is that, since $p>2$, we cannot directly get \eqref{alp4} via H\"older's inequality. Instead, we will exploit Poincar\'e inequality on annuli Lemma \ref{annuli:lemma}, which in conjunction with \eqref{alp3} will allow us to use the hole filling technique of Widman as in \cite[Eq. (5.15)]{accfm}, giving a control on the $L^2(\Omega;\vrho)$-norm of $\nabla \A(\nabla u)$ only via the $L^1$-norm of $\A(\nabla u)$, which in turn can be estimated via H\"older's inequality  by the $W^{1,p}(\Omega;\vrho)$-norm of $u$.
\\

\emph{Step 1.} Here we assume that \eqref{f:cinf} is valid. Let $u_m,\ume$ be the solutions to \eqref{sol:umloc} and \eqref{equa:locue} respectively, with the only difference that $\vrho$ is regularized in the following way:
\begin{equation}
    g_\e(x)=g\ast \eta_\e(x),\quad \rme(x)=[g_\e(x)]^a\,.
\end{equation}

We fix $R>0$ small enough, $x_0\in \mathrm{Graph}\,F$,  and take a sequence  $\{x_{0,m}\}\subset \mathrm{Graph}\,F_m$ such that $x_{0,m}\xrightarrow{m\to\infty} x_0$. We also write $\mathcal{R}^F_{\d}(x_{0,m})=x_{0,m}+\mathcal{R}^F_{\d}$, where $\mathcal{R}^F_{\d}$ is defined by \eqref{def:rectangles}.
Given two constants $\sigma,\tau$ such that $0<R\leq \s<\tau\leq 2R$, we take a cut-off function $\vphi\in C^\infty_c\big(\mathcal{R}^F_{\tau}(x_{0,m})  \big)$ such that
\begin{equation}\label{cutoff}
    \vphi\equiv 1\quad\text{in $\mathcal{R}^F_{\sigma}(x_{0,m})$},\quad |\nabla \vphi|\leq \frac{8}{(\tau-\sigma)}\,.
\end{equation}
By arguing as in \eqref{alp}-\eqref{alp2}, and using the properties of $\vphi$ \eqref{cutoff}, we find
\begin{align}\label{alp5}
        \int_{\Omega_{m,R}\cap \mathcal{R}^F_{\sigma}(x_{0,m})}\rme\,|\nabla \A_\e(\nabla \ume)|^2\,dx\leq & \,C\,\int_{\Omega_{m,R}\cap \mathcal{R}^F_{\tau}(x_{0,m})}\rme\,f^2\,\,dx
        \\
        &+\frac{C}{(\tau-\sigma)^2}\,\int_{\Omega_{m,R}\cap \big(\mathcal{R}^F_{\tau}(x_{0,m})\setminus \mathcal{R}^F_{\sigma}(x_{0,m})\big)}\rme\,|\A_\e(\nabla \ume)|^2\,dx\,.\nonumber
\end{align}
for some constant $C$ depending only on $n$ and $p$. On the other hand, by using Equation \eqref{poinc:annuli} with the choice of $\d=(\tau-\s)/(10\,n\,(1+L_F))$, we have that
\begin{align}\label{alp6}
        \frac{C}{(\tau-\sigma)^2}\,\int_{\Omega_{m,R}\cap \big(\mathcal{R}^F_{\tau}(x_{0,m})\setminus \mathcal{R}^F_{\sigma}(x_{0,m})\big)} & \rme\,|\A_\e(\nabla \ume)|^2\,dx 
        \\
        & \leq \,C_1\,\int_{\Omega_{m,R}\cap \big(\mathcal{R}^F_{\tau}(x_{0,m})\setminus \mathcal{R}^F_{\sigma}(x_{0,m})\big)}\rme\,|\nabla \A_\e(\nabla \ume)|^2\,dx
        \\
        &+\frac{C_1}{(\tau-\sigma)^{n+2+a}}\left(\int_{\Omega_{m,R}\cap \big(\mathcal{R}^F_{\tau}(x_{0,m})\setminus \mathcal{R}^F_{\sigma}(x_{0,m})\big)}\rme\,|\A_\e(\nabla \ume)|\,dx\right)^2\,.\nonumber
\end{align}
with $C_1=C_1(n,\vrho,L_F)$. Coupling the inequalities \eqref{alp5}-\eqref{alp6}, we find
\begin{align}
        \int_{\Omega_{m,R}\cap \mathcal{R}^F_{\sigma}(x_{0,m})}\rme\,|\nabla \A_\e(\nabla \ume)|^2\,dx\leq \,C\,\int_{\Omega_{m,R}\cap \mathcal{R}^F_{\tau}(x_{0,m})}\rme\,f^2\,\,dx
        \\
         + \,C_1\,\int_{\Omega_{m,R}\cap \big(\mathcal{R}^F_{\tau}(x_{0,m})\setminus \mathcal{R}^F_{\sigma}(x_{0,m})\big)}\rme\,|\nabla \A_\e(\nabla \ume)|^2\,dx\nonumber
        \\
        +\frac{C_1}{(\tau-\sigma)^{n+2+a}}\left(\int_{\Omega_{m,R}\cap \big(\mathcal{R}^F_{\tau}(x_{0,m})\setminus \mathcal{R}^F_{\sigma}(x_{0,m})\big)}\rme\,|\A_\e(\nabla \ume)|\,dx\right)^2.\nonumber
    \end{align}
After adding the quantity $C_1\,\int_{\Omega_{m,R}\cap \mathcal{R}^F_{\sigma}(x_{0,m})}\rme\,|\nabla \A_\e(\nabla \ume)|^2\,dx$ to both sides of the last inequality, we get
\begin{align}\label{alp7}
\int_{\Omega_{m,R}\cap \mathcal{R}^F_{\sigma}(x_{0,m})}&\,\rme\,|\nabla \A_\e(\nabla \ume)|^2\,dx\leq \,\frac{C}{1+C_1}\,\int_{\Omega_{m,R}\cap \mathcal{R}^F_{\tau}(x_{0,m})}\rme\,f^2\,\,dx
        \\
         &+ \,\frac
    {C_1}{1+C_1}\,\int_{\Omega_{m,R}\cap \big(\mathcal{R}^F_{\tau}(x_{0,m})\big)}\rme\,|\nabla \A_\e(\nabla \ume)|^2\,dx\nonumber
        \\
        &+\frac{C_1}{(1+C_1)(\tau-\sigma)^{n+2+a}}\left(\int_{\Omega_{m,R}\cap \big(\mathcal{R}^F_{2R_0}(x_{0,m})\setminus \mathcal{R}^F_{R_0}(x_{0,m})\big)}\rme\,|\A_\e(\nabla \ume)|\,dx\right)^2\,.\nonumber
\end{align}
A standard iteration lemma \cite[Lemma 6.1, pp. 191]{giusti} enables us to deduce from \eqref{alp7} that
\begin{align}\label{alp8}
    \int_{\Omega_{m,R}\cap \mathcal{R}^F_{R_0}(x_{0,m})}\rme\,|\nabla \A_\e(\nabla \ume)|^2\,dx \leq & \, C\,\int_{\Omega_{m,R}\cap \mathcal{R}^F_{2R_0}(x_{0,m})}\rme\,f^2\,\,dx
    \\
    &+\frac{C}{R_0^{n+2+a}}\left(\int_{\Omega_{m,R}\cap \big(\mathcal{R}^F_{2R_0}(x_{0,m})\setminus \mathcal{R}^F_{R_0}(x_{0,m})\big)}\rme\,|\A_\e(\nabla \ume)|\,dx\right)^2\,,\nonumber
\end{align}
for some constant $C$ depending on $n,p,\vrho,\Omega$.

By letting $\e\to 0$ in \eqref{alp8}, and by using the properties \eqref{conv:vrhom}, \eqref{conv:Ae} and \eqref{cisiamm}, we get
\begin{align}\label{alp9}
         \int_{\Omega_{m,R}\cap \mathcal{R}^F_{R_0}(x_{0,m})}\vrho\,|\nabla \A(\nabla u_m)|^2\,dx \leq & C\,\int_{\Omega_{R}\cap \mathcal{R}^F_{2\,R_0}(x_{0,m})}\vrho\,f^2\,\,dx
    \\
    &+\frac{C}{R_0^{n+2+a}}\left(\int_{\Omega_{m,R}\cap \big(\mathcal{R}^F_{2R_0}(x_{0,m})\setminus \mathcal{R}^F_{R_0}(x_{0,m})\big)}\vrho\,|\A(\nabla u_m)|\,dx\right)^2\,.\nonumber
\end{align}
Then, by H\"older's inequality and \eqref{stimaloc:enm}, we find
\begin{align}\label{alp10}
    \Bigg(\int_{\Omega_{m,R}\cap \big(\mathcal{R}^F_{2R_0}(x_{0,m})\setminus \mathcal{R}^F_{R_0}(x_{0,m})\big)}& \vrho\,|\A(\nabla u_m)|\,dx\Bigg)^2\leq C_0\,\left(\int_{\Omega_{m,R}\cap \big(\mathcal{R}^F_{2R_0}(x_{0,m})\setminus \mathcal{R}^F_{R_0}(x_{0,m})\big)}\vrho\, |\nabla u_m|^p\,dx\right)^{2/p'}
    \\
    &\leq C'_0\,\left(\int_{\Omega_R}\left( |u|^p+|\nabla u|^p\right)\,\vrho\,dx\right)^{2/p'}+C'_1\,\left(\int_{\Omega_R} |f|^{p'}\,\vrho\,dx\right)^{2/p'}\,.\nonumber
\end{align}
with $C'_0,C'_1$ depending on $n,p,\Omega,\vrho$, but independent on $m$.  Equations \eqref{alp9}-\eqref{alp10} provide a $W^{1,2}(\Omega_{m,R}\cap \mathcal{R}^F_{R_0}(x_{0,m});\vrho)$-estimate for the stress field $\A(\nabla u_m)$, uniform with respect to $m\in \N$. Thereby, recalling that $\Omega_{m,R}\nearrow \Omega_R$, that $x_{0,m}\to x_0\in \mathrm{Graph}\,F$, and the property  \eqref{conve2}, by letting $m\to \infty$ in \eqref{alp9}-\eqref{alp10} we finally obtain
\begin{align}
                 \int_{\Omega_{R}\cap \mathcal{R}^F_{R_0}(x_{0})}\vrho\,|\nabla \A(\nabla u)|^2\,dx \leq &\, C\,\int_{\Omega_{R}\cap \mathcal{R}^F_{2\,R_0}(x_{0})}\vrho\,f^2\,\,dx
    \\
    &+\frac{C}{R_0^{n+2+a}}\left(\int_{\Omega_{R}\cap \big(\mathcal{R}^F_{2R_0}(x_{0})\setminus \mathcal{R}^F_{R_0}(x_{0})\big)}\vrho\,|\A(\nabla u)|\,dx\right)^2\,,\nonumber
\end{align}
and finally by H\"older's inequality, we obtain
\begin{align}
         \Bigg(\int_{\Omega_{R}\cap B_{R_0}(x_0)} \vrho\,|\A(\nabla u)|\,dx\Bigg)^2\leq \left( \int_{\Omega_R\cap B_{R_0}(x_0)}\vrho\,dx\right)^{2/p}\,\left(\int_{\Omega_R\cap B_{R_0}(x_0)}|\nabla u|^p\,\vrho\,dx\right)^{2/p'}\,,\nonumber
\end{align}
where $C$ depends on $n,p,\Omega,\vrho$, that is our thesis.
\\

\noindent\emph{Step 2.} In this step we remove assumption \eqref{f:cinf}. This step agrees with the proof of Step 3 of Theorem \ref{thm:loc1}, with the only difference that one needs to consider rectangles $\mathcal{R}^F_{R_0}(x_0)$ in place of balls. The details are omitted.

\end{proof}

\section{Some remarks on the Dirichlet problem} \label{sec:dir} 

In this section we give some remark on Dirichlet problem for Equation \eqref{eq_main1}. We notice that it is possible to prove a version of Reilly's identity which would be suitable when one considers a Dirichlet problem. However, as we are going to explain below, it seems to us that the type of weight that we are considering is not suitable for studying regularity for Dirichlet problems.

Reilly's identity for Dirichlet problems can be obtained as follows. We first notice that will make use of the following identity, whose proof can be found in \cite[Lemma 4.3]{accfm} for more general operators.
First observe that, for $V\in C^2(\overline{\Omega})$, by subtracting
\begin{equation*}
     \mathrm{div} V\,(V\cdot \nu)=\mathrm{div}_T V\,(V\cdot \nu)+(\partial_\nu V\cdot \nu)\,(V\cdot \nu)\,,
\end{equation*}
with
\begin{equation*}
    \nabla V\,V\cdot \nu=\nabla_T V\,V_T\cdot \nu+(\partial_\nu V)\cdot\nu\,(V\cdot \nu)\,,
\end{equation*}
we get
\begin{equation}\label{acc}
    \mathrm{div} V\,(V\cdot \nu)-\nabla V\,V\cdot \nu=\mathrm{div}_T V\,V\cdot \nu-\nabla_T V\,V_T\cdot \nu\,.
\end{equation}
Now assume that $V$ is of the form $V(x)=\mathfrak a(x)\,\nabla v$, for two function $\mathfrak a\in C^2(\overline{\Omega})$, and $v\in C^2(\overline{\Omega})$. Owing to \cite[Lemma 4.3]{accfm} and \eqref{acc}, if $v=0$ on $\partial \Omega$ there holds
\begin{align}\label{re:temp3}
\mathrm{div}\big(\mathfrak a\,\nabla v \big) \,\mathfrak a\,\nabla v\cdot \nu
   &-a\, \nabla\big[ \mathfrak a\,\nabla v\big] \,\nabla_T v\cdot \nu
   \\
\mathrm{div}_T\big(\mathfrak a\,\nabla v \big) \,\mathfrak a\,\nabla v\cdot \nu
   &-a\, \nabla_T\big[ \mathfrak a\,\nabla v\big] \,\nabla_T v\cdot \nu\nonumber
    \\ \nonumber
    & =\mathfrak a^2\,|\nabla v|^2\,\mathrm{tr}\,\mathcal{B}\nonumber
\end{align}
on $\partial\Omega$. Here $\mathrm{tr}\,\mathcal{B}$ is the mean curvature of $\partial \Omega$.

As an immediate consequence we obtain the following proposition.
\begin{proposition}\label{prop:funddir}
    Let $\Omega$ be a bounded domain of class $C^2$. Let $\mathfrak a\in C^2(\overline{\Omega})$, $h\in C^2(\overline{\Omega})$, $\vrho=e^{-h}$, and assume that $u\in C^2(\overline{\Omega})$ vanishes on $\partial \Omega$. Then
    \begin{align}\label{id:ddiirr}
             \int_\Omega \vrho^{-1}\big[ \mathrm{div}\big(\vrho\,\mathfrak a\,\nabla u\big)\big]^2\,dx=\int_\Omega& \vrho\,\mathrm{tr}\big[\big( \nabla(\mathfrak a\,\nabla u)\big)^2\big]\,dx+\int_{\partial \Omega}\vrho\,\mathfrak a^2\,|\nabla u|^2\,\mathrm{tr}\,\B\,d\H^{n-1}+\int_\Omega\vrho\,\mathfrak{a}^2\,\nabla^2 h\,\nabla u\cdot \nabla u\,dx
             \\
             &+\int_{\partial \Omega} \mathfrak{a}^2\, (\partial_\nu\vrho)\,|\nabla u|^2\,d\H^{n-1}\,.\nonumber
    \end{align}
\end{proposition}

Unfortunately, \eqref{id:ddiirr} is not helpful whenever the weight $\vrho$ vanishes at the boundary, since the last term in \eqref{id:ddiirr} can be difficult to be bounded. 

For instance, if $\vrho(x)=\mathrm{dist}(x,\partial \Omega)^a=d(x)^a$, $a>0$, then $\vrho$ is $A_2$-Muckenhoupt when $a \in (0,1)$, and in this case one has a trace theorem. On the other hand, in this case $\partial_\nu\vrho\equiv +\infty$ on $\partial \Omega$ which makes impossible the chance of using a trace theorem and obtain uniform upper bounds in terms of the desired quantities. This is consistent with the example for Dirichlet problem given in the Introduction.

\section*{Acknowledgments} 
\noindent The authors have been partially supported by the ``Gruppo Nazionale per l'Analisi Matematica, la Probabilit\`a e le loro Applicazioni'' (GNAMPA) of the ``Istituto Nazionale di Alta Matematica'' (INdAM, Italy) and by the Research Project of the Italian Ministry of University and Research (MUR) Prin 2022 
``Partial differential equations and related geometric-functional inequalities'', grant number 20229M52AS\_004.

\end{document}